\documentclass{amsart}

\usepackage{fullpage}

\usepackage[all]{xy}

\theoremstyle{plain}
\newtheorem{bigthm}{Theorem}[section]
\newtheorem{bigprop}[bigthm]{Proposition}
\newtheorem{biglemma}[bigthm]{Lemma}
\newtheorem{bigcor}[bigthm]{Corollary}

\newtheorem{thm}{Theorem}[subsection]

\newtheorem{lemma}[thm]{Lemma}

\theoremstyle{definition}

\newtheorem{bigexample}[bigthm]{Example}

\theoremstyle{remark}
\newtheorem{bigrem}[bigthm]{Remark}

\newcommand{\HH}{\mathrm{H}}

\newcommand{\SD}{S\!D}
\newcommand{\Q}{GQ}
\newcommand{\D}{\mathrm{D}}
\newcommand{\Tr}{\mathrm{Tr}}

\begin{document}

\title[Universal deformation rings]{Universal deformation rings, endo-trivial modules, and semidihedral and generalized quaternion $2$-groups}

\author{Frauke M. Bleher}
\address{F.B.: Department of Mathematics\\University of Iowa\\
14 MacLean Hall\\Iowa City, IA 52242-1419, U.S.A.}
\email{frauke-bleher@uiowa.edu}
\thanks{The first author was supported in part by NSF FRG Grant No.\ DMS-1360621.}
\author{Ted Chinburg}
\address{T.C.: Department of Mathematics\\University of Pennsylvania\\Philadelphia, PA 19104, U.S.A.}
\email{ted@math.upenn.edu}
\thanks{The second author was supported in part by NSF FRG Grant No.\ DMS-1360767, NSF FRG Grant No.\ DMS-1265290,
NSF SaTC grant No. CNS-1513671  and Simons Foundation Grant No.\ 338379.}
\author{Roberto C. Soto}
\address{R.S.: Department of Mathematics\\College of Natural Sciences \& Mathematics\\
California State University, Fullerton\\Fullerton, CA 92834-6850, U.S.A.}
\email{rcsoto@fullerton.edu}

\subjclass[2000]{Primary 20C20; Secondary 20C15, 16G10}
\keywords{Universal deformation rings, endo-trivial modules, stable endomorphism rings, semidihedral $2$-groups, generalized quaternion $2$-groups}

\begin{abstract}
Let $k$ be a field of characteristic $p>0$, and let $W$ be a complete discrete valuation ring of characteristic $0$ that has $k$ as its
residue field. Suppose $G$ is a finite group and $G^{\mathrm{ab},p}$ is its maximal abelian $p$-quotient group.
We prove that every endo-trivial $kG$-module $V$ has a universal deformation ring that is isomorphic to 
the group ring $WG^{\mathrm{ab},p}$. In particular, this gives a positive answer to a question raised by Bleher and Chinburg for all endo-trivial modules.
Moreover, we show that the universal deformation of $V$ over
$WG^{\mathrm{ab},p}$ is uniquely determined by any lift of $V$ over $W$. 
In the case when $p=2$ and $G=\D$ is a $2$-group that is either semidihedral or generalized quaternion,
we give an explicit description of the universal deformation of every indecomposable endo-trivial $k\D$-module $V$.
\end{abstract}

\maketitle


\section{Introduction}
\label{s:intro}
\setcounter{equation}{0}

Let $k$ be a field of positive characteristic $p$, let $G$ be a finite group, and let $V$
be a finitely generated $kG$-module. It is a classical problem 
to analyze when $V$ can be lifted to a module for $G$ over a complete discrete valuation ring $W$ of characteristic 0
that has $k$ as its residue field. For example, Green showed in
\cite{green} that if all 2-extensions of $V$ by itself are trivial, then $V$ can always be lifted over $W$. 
When $G$ is a $p$-group, Alperin showed in \cite{alpendotrivial} that every endo-trivial $kG$-module can be
lifted to an endo-trivial $WG$-module.  In \cite{malle}, Alperin's argument was generalized to 
endo-trivial modules for arbitrary finite groups 
under the assumption that $k$ is algebraically closed. Using a simple base change argument, one can
in fact drop the assumption that $k$ is algebraically closed (see Lemma \ref{lem:improvemalle}).
The question of whether $V$ can be lifted
over $W$ can be reformulated in terms of the versal deformation ring $R(G,V)$ of $V$ 
by asking whether there is a $W$-algebra homomorphism $R(G,V)\to W$. 
It is then a natural problem
to determine the full versal deformation ring $R(G,V)$ and the corresponding versal deformation. 
In this paper, we do this for all endo-trivial $kG$-modules $V$. 
In the case when $k$ has characteristic $2$ and $G=\D$ is either a semidihedral or a generalized quaternion $2$-group,
we moreover give an explicit description of the versal deformation of every indecomposable endo-trivial $k\D$-module.

Recall that for arbitrary $p$ and $G$, a finitely generated $kG$-module $V$ is called endo-trivial if the $kG$-module
$\mathrm{Hom}_k(V,V)\cong V^*\otimes_k V$ is isomorphic to a direct
sum of the trivial simple $kG$-module $k$ and a projective $kG$-module. 
Endo-trivial modules play an important role
in the modular representation theory of finite groups. They arise in the context of 
derived equivalences and stable equivalences of block algebras, and also as building blocks for
the more general endo-permutation modules (see e.g. \cite{dade,thevenaz}). 
In \cite{carl1.5,carl2}, Carlson and Th\'{e}venaz classified all endo-trivial $kG$-modules
when $G$ is a $p$-group. The classification of the endo-trivial $kG$-modules for arbitrary finite groups is still ongoing 
(see e.g. \cite{carlsonpapers} and its references). Since the stable endomorphism ring of every endo-trivial $kG$-module
is one-dimensional over $k$, it follows
from \cite[Prop. 2.1]{bc} that $V$ has a well-defined universal deformation ring $R(G,V)$.

The topological ring $R(G,V)$ is universal with respect to deformations of $V$ over all complete local 
commutative Noetherian $W$-algebras $R$ with residue field $k$. A deformation of $V$ over such a ring
$R$ is given by the isomorphism class of a finitely generated $RG$-module $M$ which is free
over $R$, together 
with a $kG$-module isomorphism $k\otimes_R M\to V$ (see Section \ref{s:prelim} for more details). 

In number theory, the main motivation for studying universal
deformation rings for finite groups is to provide evidence for and counter-examples to various 
possible conjectures concerning ring theoretic properties of universal deformation rings for 
profinite Galois groups. The idea is that universal deformation rings for finite groups can be more easily
described using deep results from modular representation theory
due to Brauer,  
Erdmann \cite{erd}, Linckelmann \cite{linckel,linckel2}, and others.
Moreover, the results in \cite{lendesmit} show that if $\Gamma$ is a profinite group and
$V$ is a finite dimensional $k$-vector space with a continuous $\Gamma$-action
which has a universal deformation ring, then $R(\Gamma,V)$ is
the inverse limit of the universal deformation rings $R(G,V)$ when 
$G$ ranges over all finite discrete quotient groups of $\Gamma$ through which the $\Gamma$-action on $V$ factors. Thus to answer questions about 
the ring structure of $R(\Gamma,V)$, it is natural to first consider the case when $\Gamma=G$ is finite. 

\medskip

The following is our first main result.

\begin{bigthm}
\label{thm:supermain}
Let $V$ be a finitely generated endo-trivial $kG$-module.  Let $V_1$ be a non-projective indecomposable direct summand of $V$, which is unique up to isomorphism. 
\begin{itemize}
\item[(i)] The universal deformation ring $R(G,V)$ is isomorphic to the group algebra $WG^{\mathrm{ab},p}$ 
where $G^{\mathrm{ab},p}$ is the maximal abelian 
$p$-quotient group of $G$. 
\item[(ii)] If $V_W$ is a lift of $V$ over $W$, then the isomorphism class of $V_W\otimes_WWG^{\mathrm{ab},p}$,
with diagonal $G$-action, is the universal deformation of $V$ over $WG^{\mathrm{ab},p}$.
\item[(iii)]
Suppose $D$ is a defect group of the block of $kG$ to which $V_1$ belongs. Then $R(G,V)$ is a subquotient ring of the group ring 
$WD$, giving a positive answer to \cite[Question 1.1]{bc} for endo-trivial modules.
\end{itemize}
\end{bigthm}

In \cite{brouepuig,puig}, Brou\'{e} and Puig introduced and studied so-called nilpotent blocks. 
Using \cite{puig}, we obtain the following result as an easy consequence of Theorem 
\ref{thm:supermain}, where we assume as in \cite{puig} that $k$ is algebraically closed.

\begin{bigcor}
\label{cor:nilpotent}
Suppose $k$ is algebraically closed, and $B$ is a nilpotent block of $kG$. 
Suppose $V$ is a finitely generated $B$-module whose stable endomorphism ring is isomorphic to $k$.
Then the universal deformation ring $R(G,V)$ is isomorphic to $WG^{\mathrm{ab},p}$.
\end{bigcor}

The main tools in the proof of Theorem \ref{thm:supermain} are the results in \cite{alpendotrivial,malle} about lifting
endo-trivial $kG$-modules to endo-trivial $WG$-modules, together with the result \cite[Lemma 2.2.2]{bleherTAMS2009} 
that shows that stable equivalences of Morita type over $W$ preserve universal deformation rings.  

Alperin's proof in \cite{alpendotrivial}, and its extended version in \cite{malle}, 
of the existence of a lift of a given endo-trivial $kG$-module $V$ over 
$W$ is not constructive. Namely, after reducing to the case when the image of $V$ lies in $\mathrm{SL}_n(k)$,
endo-triviality of $V$ is used to show that an infinite sequence of lifting obstructions vanishes, leading to a lift of this 
image to a subgroup of  $\mathrm{SL}_n(W)$.

Since an explicit description of the universal deformation in part (ii) of Theorem \ref{thm:supermain} depends 
on an explicit description of a lift of $V$ over $W$, the question then arises how one can construct an explicit lift $V_W$ of
$V$ over $W$ without using an infinite sequence of vanishing lifting obstructions. The goal is to provide an explicit description
of the $G$-action on a particular $W$-basis of $V_W$.
This problem can be compared to constructing an explicit solution by radicals to a polynomial equation
rather than just proving that such a solution exists.

In this paper, we focus on the case when $G$ is a $p$-group and the group of equivalence classes of endo-trivial modules 
$T(G)$ has a non-trivial torsion subgroup. By \cite[Thm. 1.1]{carl1.5}, this means that either $G$ is cyclic of order at least 3,
or $p=2$ and $G=\D$ is a $2$-group that is either semidihedral or generalized quaternion.

The following is a summary of our second main result. For more detailed versions, see Example \ref{ex:cyclic}
and Propositions \ref{prop:unidefSD}, \ref{prop:unidefQ8} and \ref{prop:betaliftpowerseries} - \ref{prop:unidefQ}.

\begin{bigthm}
\label{thm:dontknowhowtosay}
Let $G$ be a $p$-group such that the group of equivalence classes of endo-trivial modules 
$T(G)$ has a non-trivial torsion subgroup. Let $V$ be an indecomposable endo-trivial $kG$-module,
and let $V_W$ be a lift of $V$ over $W$.

To 
provide an explicit description of the $G$-action on a particular $W$-basis of $V_W$,
one needs to solve at most one quadratic equation of the form
\begin{equation}
\label{eq:goodeq}
b^2 + f_1(t)\,b + f_0(t) = 0
\end{equation}
with $f_0(t),f_1(t)\in\mathbb{Z}[t]$ for an element $b$ in $\mathbb{Z}_p[[t]]$. 
Moreover, if one needs to solve such an equation $(\ref{eq:goodeq})$ then $p=2$, and 
finding such a solution is equivalent to taking the square roots of explicitly given elements in certain
cyclotomic extensions of $\mathbb{Q}_2$.
The isomorphism class of $V_W\otimes_WWG^{\mathrm{ab},p}$,
with diagonal $G$-action, provides then an explicit universal deformation of $V$ over $WG^{\mathrm{ab},p}$.
\end{bigthm}

The main ingredient in the proof of Theorem \ref{thm:dontknowhowtosay} is the explicit description of the indecomposable
endo-trivial $kG$-modules $V$, up to isomorphism. If $G$ is cyclic, this is given in \cite[Cor. 8.8]{dade}. If $p=2$ and $G=\D$ is either a 
semidihedral or a generalized quaternion 2-group, we use the description of the indecomposable endo-trivial $k\D$-modules provided 
by Carlson and Th\'{e}venaz in \cite{CT}. If $G$ is cyclic, respectively if $G=Q_8$ is a quaternion group of order 8 and $k$ does not contain a
primitive cube root of unity, it follows that all indecomposable endo-trivial $kG$-modules lie in the syzygy orbit of the trivial simple 
$kG$-module. In all other cases, there is a second syzygy orbit of a cyclic endo-trivial $k\D$-module $L$ of $k$-dimension
$2^{d-1}-1$. The restriction of $L$ to a maximal cyclic subgroup $X=\langle x\rangle$ of  $\D$ is isomorphic to 
$\overline{S}=kX/(k\cdot \Tr_X)$, where $\Tr_X$ is the sum of the elements in $X$. Let $S=WX/ (W\cdot \Tr_X)$, and
let $\sigma_x$ denote the $W$-linear automorphism of $S$ given by multiplication by $x$.
Since $\D$ is generated by $x$ and any element $y$ that does not lie in $X$, finding a lift of $L$ over $W$ is equivalent to
constructing a $W$-linear automorphism $\sigma_y$ of $S$ such that $\sigma_x$ and $\sigma_y$ satisfy
the relations of $x$ and $y$ in $\D$ and such that  $\sigma_y$ reduces to the action of $y$ on $\overline{S}$
modulo the maximal ideal of $W$. If $\D$ is semidihedral, respectively if $\D=Q_8$ and $k$ contains a primitive cube root of unity, 
we are able to provide directly an explicit description of $\sigma_y$. On the other hand,
if $\D$ is generalized quaternion of order at least 16, then we need to
find a solution $b\in\mathbb{Z}_2[[t]]$ of a quadratic equation of the form $(\ref{eq:goodeq})$ to construct an automorphism $\sigma_y$.
The quadratic equation results from the fact that $y$ has order 4 and satisfies the relation $yxy^{-1}=x^{-1}$ 
but that the automorphism of $S$ resulting from inverting the elements of $X$ has only order 2. 

The group algebras $k\D$ are examples of blocks of tame representation type, which were
studied in great detail by Erdmann in \cite{erd}. In addition to providing all possible quivers and relations for
these tame block algebras, she also described their stable Auslander-Reiten quivers. It is a natural problem
to determine the location of the isomorphism classes of indecomposable endo-trivial modules in these
stable Auslander-Reiten quivers. In this paper, we do this for $k\D$.
We show that the indecomposable endo-trivial $k\D$-modules
lie in four $\Omega^2$-orbits that are either at the ends of two non-periodic components of type $D_\infty$ if 
$\D$ is semidihedral or at the ends of four 2-tubes if $\D$ is generalized quaternion
(see Lemmas \ref{cor:ARSD} and \ref{cor:ARQ}).
The main ingredient in the proof is a slight variation
of results in \cite{bondarenkodrozd} and \cite{Dade1} to obtain an explicit isomorphism
between $k\D$ and a $k$-algebra $\Lambda_{\D}$ of the form $k\langle a,b\rangle/I_{\D}$ such that the
socle scalars match the descriptions in \cite[Thm. III.1 and Sect. III.13]{erd}. 
We also give precise descriptions of the $\Lambda_{\D}$-modules corresponding to indecomposable endo-trivial 
$k\D$-modules (see Lemmas \ref{lem:YLSD} and \ref{lem:LQ}).
In particular, this provides a different approach to endo-trivial $k\D$-modules.

\medskip

The paper is organized as follows:
In Section \ref{s:prelim}, we give a brief introduction to deformation rings and deformations. 
In particular, we show in Remark \ref{rem:syzygylift} how to use an explicit lift of a finitely generated $kG$-module $V$ over $W$ to
construct an explicit lift of any (co-)syzygy $\Omega^i(V)$ over $W$. 
In Section \ref{s:endotrivial}, we show in Lemma \ref{lem:improvemalle} that we can lift any endo-trivial $kG$-module $V$
over $W$, without any additional assumptions on $k$. Moreover,  we prove Theorem
\ref{thm:supermain} and Corollary \ref{cor:nilpotent}, and we discuss the case of cyclic $p$-groups in Example \ref{ex:cyclic}.  In Sections \ref{s:semidih}, \ref{s:quat8} and \ref{s:quat}, 
we prove Theorem \ref{thm:dontknowhowtosay} by giving an explicit description
of the universal deformation of every indecomposable endo-trivial $k\D$-modules $V$ when $p=2$ and $\D$ is a semidihedral
or generalized quaternion $2$-group; see Propositions \ref{prop:unidefSD}, \ref{prop:unidefQ8} and \ref{prop:betaliftpowerseries} - \ref{prop:unidefQ}. 
Finally, in Section \ref{s:appendix}, we determine the location of the isomorphism classes of the
indecomposable endo-trivial $k\D$-modules in the stable Auslander-Reiten quiver of $k\D$;
see Lemmas \ref{cor:ARSD} and \ref{cor:ARQ}.

Part of this paper constitutes the Ph.D. thesis of the third author under the supervision
of the first author \cite{soto}.

\medskip

Unless stated otherwise, our modules are finitely generated left modules. On the other hand, when $R$ and $S$ are associative
rings with 1 then an $R$-$S$-bimodule $M$ is both a left $R$-module and a right $S$-module such that for all $r\in R$, $s\in S$ and
$m\in M$ one has $r(ms)=(rm)s$. Our maps are written on the left such that the map composition $f\circ g$
means $f$ after $g$.

\medskip
\noindent
\textbf{Acknowledgements.} We would like to thank the referee of an earlier version of this paper for very helpful comments 
which  improved the results of this paper.


\section{Preliminaries}
\label{s:prelim}
\setcounter{equation}{0}

In this section, we give a brief introduction to versal and universal deformation rings and deformations. 
For more background material, we refer the reader to \cite{maz1} and \cite{lendesmit}.

Let $k$ be a field of positive characterstic $p$, and let $W$ be a complete discrete valuation ring of characteristic 0 that has $k$ as its residue field.
If $k$ is a perfect field, one often chooses $W$ to be the ring of infinite Witt vectors over $k$.
Let $\hat{\mathcal{C}}$ be the category of all complete local commutative Noetherian $W$-algebras $R$ with unique maximal ideal $m_R$
and fixed residue map $\pi_R:R\to k=R/m_R$. The morphisms in $\hat{\mathcal{C}}$ are continuous $W$-algebra 
homomorphisms which induce the identity map on $k$.

Suppose $G$ is a finite group and $V$ is a finitely generated $kG$-module. 
A lift of $V$ over an object $R$ in $\hat{\mathcal{C}}$ is a pair $(M,\phi)$ where $M$ is a finitely 
generated $RG$-module that is free over $R$, and $\phi:k\otimes_R M\to V$ is an isomorphism of 
$kG$-modules. Two lifts $(M,\phi)$ and $(M',\phi')$ of $V$ over $R$ are said to be isomorphic if there is an $RG$-module
isomorphism $f:M\to M'$ with $\phi=\phi'\circ (k\otimes f)$. The isomorphism class $[M,\phi]$ of a lift 
$(M,\phi)$ of $V$ over $R$ is called a deformation of $V$ over $R$, and the set of all such deformations 
is denoted by $\mathrm{Def}_G(V,R)$. The deformation functor
$$\hat{F}_V:\hat{\mathcal{C}} \to \mathrm{Sets}$$ 
is a covariant functor which
sends an object $R$ in $\hat{\mathcal{C}}$ to $\mathrm{Def}_G(V,R)$ and a morphism 
$\alpha:R\to R'$ in $\hat{\mathcal{C}}$ to the map $\mathrm{Def}_G(V,R) \to
\mathrm{Def}_G(V,R')$ defined by $[M,\phi]\mapsto [R'\otimes_{R,\alpha} M,\phi_\alpha]$, where  
$\phi_\alpha=\phi$ after identifying $k\otimes_{R'}(R'\otimes_{R,\alpha} M)$ with $k\otimes_R M$.

Suppose there exists an object $R(G,V)$ in $\hat{\mathcal{C}}$ and a deformation 
$[U(G,V),\phi_U]$ of $V$ over $R(G,V)$ with the following property:
For each $R$ in $\hat{\mathcal{C}}$ and for each lift $(M,\phi)$ of $V$ over $R$ there exists 
a morphism $\alpha:R(G,V)\to R$ in $\hat{\mathcal{C}}$ such that $\hat{F}_V(\alpha)([U(G,V),\phi_U])=
[M,\phi]$, and moreover $\alpha$ is unique if $R$ is the ring of dual numbers
$k[\epsilon]$. Then $R(G,V)$ is called the versal deformation ring of $V$ and 
$[U(G,V),\phi_U]$ is called the versal deformation of $V$. If the morphism $\alpha$ is
unique for all $R$ and all lifts $(M,\phi)$ of $V$ over $R$, 
then $R(G,V)$ is called the universal deformation ring of $V$ and $[U(G,V),\phi_U]$ is 
called the universal deformation of $V$. In other words, $R(G,V)$ is universal if and only if
$R(G,V)$ represents the functor $\hat{F}_V$ in the sense that $\hat{F}_V$ is naturally isomorphic to 
the Hom functor
$\mathrm{Hom}_{\hat{\mathcal{C}}}(R(G,V),-)$. 

By \cite{maz1}, every finitely generated $kG$-module $V$ has a versal deformation ring.
By a result of Faltings (see \cite[Prop. 7.1]{lendesmit}), $V$ has a universal deformation ring if 
$\mathrm{End}_{kG}(V)=k$.

\begin{bigprop} {\rm (\cite[Prop. 2.1]{bc}, \cite[Rem. 2.1]{3quat})}
\label{prop:blchin} 
Suppose $V$ is a finitely generated $kG$-mod\-ule whose stable endomorphism ring 
$\underline{\mathrm{End}}_{kG}(V)$ is isomorphic to $k$. Then $V$ has a 
universal deformation ring $R(G,V)$. Moreover, if $R$ is in $\hat{\mathcal{C}}$ and
$(M,\phi)$ and $(M',\phi')$ are lifts of $~V$ over $R$ such that $M$ and $M'$ are
isomorphic as $RG$-modules then $[M,\phi]=[M',\phi']$.
\end{bigprop}

In particular, this means that if the stable endomorphism ring of $V$ is isomorphic to $k$, then
we do not need to distinguish between deformations $[M,\phi]$ of $V$ over an object $R$ in $\hat{\mathcal{C}}$,
as defined above, and so-called {\it weak} deformations $[M]$ of $V$ over $R$ (i.e., isomorphism classes of $RG$-modules
$M$ with  $k\otimes_R M\cong V$, without specifying a particular $kG$-module isomorphism $\phi:k\otimes_R M \to V$).

Moreover, we have the following result, where $\Omega$ denotes the syzygy functor (also called Heller operator,
see, for example, \cite[\S 20]{alp}).

\begin{biglemma} 
\label{lem:defhelp}
{\rm \cite[Cors. 2.5 and 2.8]{bc}}
Let $V$ be a finitely generated $kG$-module with $\underline{\mathrm{End}}_{kG}(V)\cong k$.
\begin{enumerate}
\item[(i)] Then $\underline{\mathrm{End}}_{kG}(\Omega(V))\cong k$, and $R(G,V)$ and $R(G,\Omega(V))$ 
are isomorphic.
\item[(ii)] There is a non-projective indecomposable $kG$-module $V_1$ $($unique up to
iso\-mor\-phism$)$ such that $\underline{\mathrm{End}}_{kG}(V_1)\cong k$, $V$ is isomorphic to 
$V_1\oplus Q$ for some projective $kG$-module $Q$, and $R(G,V)$ and $R(G,V_1)$ are 
isomorphic.
\end{enumerate}
\end{biglemma}

One of our goals is to explicitly construct certain lifts. The following remark shows how to construct an explicit lift of an 
arbitrary (co-)syzygy of a finitely generated $kG$-module $V$ over $W$ provided one knows an explicit lift of $V$ over $W$.

\begin{bigrem}
\label{rem:syzygylift}
Let $V$ be a finitely generated $kG$-module, 
and suppose we know an explicit lift $V_W$ of $V$ over $W$, with corresponding $kG$-module isomorphism
$\phi_V:k\otimes_W V_W\to V$.

By \cite[Props. (6.5) and (6.7)]{CR}, a projective $kG$-module can always be lifted over $W$. 
Let $\pi:P\to V$ (resp. $\iota:V\to E$) be a projective cover (resp. injective hull) of $V$ as a $kG$-module.
Since $kG$ is self-injective, it follows that $E$ is a projective $kG$-module.
Let $P_W$ (resp. $E_W$) be a lift of $P$ (resp. $E$) over $W$, with corresponding $kG$-module isomorphism
$\phi_P:k\otimes_W P_W\to P$ (resp. $\phi_E:k\otimes_W E_W\to E$). In particular, $E_W$ is a projective $WG$-module.

Since $P_W$ is a projective $WG$-module, there exists a $WG$-module homomorphism $\pi_W:P_W\to V_W$
such that $\phi_V\circ (k\otimes_W\pi_W)\circ\phi_P^{-1} = \pi$. By Nakayama's Lemma, it follows that $\pi_W$ is surjective.
Since $\mathrm{Ker}(\pi_W)$ is a free $W$-module, it follows that $\mathrm{Ker}(\pi_W)$ is a 
lift of $\Omega(V)=\mathrm{Ker}(\pi)$ over $W$. We use the notation
$\Omega_{WG}(V_W) = \mathrm{Ker}(\pi_W)$.

We can use a similar argument as in the proof of \cite[Prop. 2.4]{bc} to find an explicit lift of $\Omega^{-1}(V)=\mathrm{Coker}(\iota)$.
Namely, we have a short exact sequence
\begin{equation}
\label{eq:ses000}
0 \to m_W E_W \to E_W \to E \to 0
\end{equation}
and we have $\mathrm{Ext}^1_{WG}(V_W,m_W E_W)=0$.
This implies that 
there there exists a $WG$-module homomorphism $\iota_W:V_W\to E_W$
such that $\phi_E\circ (k\otimes_W\iota_W)\circ\phi_V^{-1} = \iota$. Since $\iota$ is injective and since $V_W$ and $E_W$ are free over $W$,
it follows by Nakayama's Lemma that $\iota_W$ is also injective. Thus we have a commutative diagram of $WG$-modules 
$$\xymatrix{0\ar[r]& V_W \ar[r]^{\iota_W}\ar[d]& E_W\ar[r]\ar[d] & \mathrm{Coker}(\iota_W)\ar[r]\ar[d]&0\\
0\ar[r]&V\ar[r]^{\iota} & E \ar[r]& \Omega^{-1}(V)\ar[r]&0
}$$
with exact rows,
and
$\mathrm{Coker}(\iota_W)$ is a free $W$-module. Therefore,  
$\mathrm{Coker}(\iota_W)$ is a lift of $\Omega^{-1}(V)=\mathrm{Coker}(\iota)$ over $W$. We use the notation
$\Omega^{-1}_{WG}(V_W) = \mathrm{Coker}(\iota_W)$.
\end{bigrem}


\section{Endo-trivial modules}
\label{s:endotrivial}
\setcounter{equation}{0}

Assume the notation from the previous section. Suppose $V$ is a finitely generated $kG$-module. We say $V$ is endo-trivial if 
its $k$-endomorphism ring $\mathrm{End}_k(V )$ is, as a $kG$-module, stably isomorphic to the trivial $kG$-module $k$. 
In other words, $\mathrm{End}_k(V )$ is isomorphic to a direct sum of $k$ and a projective $kG$-module.

Note that every endo-trivial $kG$-module satisfies $\underline{\mathrm{End}}_{kG}(V)\cong k$. If $G$ is a $p$-group then it 
follows from \cite{carl1} that the endo-trivial $kG$-modules coincide with the $kG$-modules $V$ with $\underline{\mathrm{End}}_{kG}(V)\cong k$.
However, for arbitrary finite groups, the set of isomorphism classes of $kG$-modules whose stable endomorphism rings are isomorphic to $k$
usually properly contains the isomorphism classes of endo-trivial $kG$-modules.

In this section, we prove Theorem \ref{thm:supermain}. One of the main ingredients in the proof is that endo-trivial $kG$-modules 
have lifts over $W$.
In the case when $G$ is a $p$-group, this was proved by Alperin in \cite{alpendotrivial} without any additional assumptions on $k$. 
For arbitrary finite groups $G$, Alperin's argument was generalized to 
arbitrary finite groups in \cite{malle}. However, the proof assumed that $k$ is algebraically closed and the fraction field $F$ of $W$ is
a splitting field of $G$ (see \cite[Sect. 2]{malle}).
We first show how to use a simple base change argument to drop these assumptions on $k$ and $W$.

\begin{biglemma}
\label{lem:improvemalle}
Let $k$ be an arbitrary field of positive characteristic $p$, and let $W$ be a complete discrete valuation ring of characteristic $0$ that has 
$k$ as its residue field. If $V$ is an endo-trivial $kG$-module then $V$ can be lifted to an endo-trivial $WG$-module $V_W$.
\end{biglemma}

\begin{proof}
We first note that \cite[Lemma 2.1]{malle} and \cite[Prop. 2.4]{malle} are true for an arbitrary field $k$ of characteristic $p$. 
Moreover, if $N$ is a $WG$-module that is free over $W$ such that $N/m_W N$ is a projective $kG$-module, then it follows, for example from \cite[Sect. 14.2]{serre}, that $N$ is a projective $WG$-module without any additional assumptions on $k$. 

Write the order of $G$ as $p^d\,m$, where $p$ does not divide $m$. 
Let $n=\mathrm{dim}_k\,V$ and let $\rho:G\to\mathrm{GL}_n(k)$ be a representation of $V$. By \cite[Lemma 2.1]{malle}, $n$ is relatively prime to $p$.

Suppose first that $k$ contains all $(nm)$-th roots of unity. Then the arguments in the proof of \cite[Thm. 1.3]{malle} show
that $V$ can be lifted to an endo-trivial $WG$-module $V_W$.

Suppose next that $k$ is an arbitrary field of characteristic $p$. Let $\zeta$ be a primitive $(nm)$-th root of unity in a 
fixed separable closure of $k$. Define $k'=k(\zeta)$ and $W'=W[\zeta]$. Note that since $nm$ is relatively prime to $p$,
$W'$ is unramified over $W$. In particular, if $\varpi$ is a uniformizer for $W$
then $\varpi$ is also a uniformizer for $W'$. Let $\rho':G\to\mathrm{GL}_n(k')$ be the representation of $G$ over $k'$ obtained
by composing $\rho$ with the inclusion $\mathrm{GL}_n(k)\hookrightarrow\mathrm{GL}_n(k')$. Define $G_\rho$ (resp. $G_{\rho'}$) to
be the image of $\rho$ (resp. $\rho'$). Our goal is to lift $G_\rho$ to a subgroup of $\mathrm{GL}_n(W)$, by knowing that we
can lift $G_{\rho'}$ to a subgroup of $\mathrm{GL}_n(W')$.  Let $\ell\ge 1$. Suppose we have lifted $G_\rho$ to a subgroup 
$G_{\rho,\ell}$ of $\mathrm{GL}_n(W/\varpi^\ell W)$ such that the injection 
$\mathrm{GL}_n(W/\varpi^\ell W)\to \mathrm{GL}_n(W'/\varpi^\ell W')$ provides a subgroup $G_{\rho',\ell}$ of
$\mathrm{GL}_n(W'/\varpi^\ell W')$ lifting $G_{\rho'}$. The obstruction to lifting $G_{\rho,\ell}$ to 
$\mathrm{GL}_n(W/\varpi^{\ell+1} W)$ is a class $c$ in $\HH^2(G,\frac{\varpi^\ell W}{\varpi^{\ell+1} W}\otimes_k \mathrm{Ad}(\rho))$,
where $\mathrm{Ad}(\rho)$ is the adjoint representation associated
to $\rho$. On the other hand, the obstruction to lifting $G_{\rho',\ell}$ to $\mathrm{GL}_n(W'/\varpi^{\ell+1} W')$ is the class
$c'$ in 
$$\HH^2(G,\frac{\varpi^\ell W'}{\varpi^{\ell+1} W'}\otimes_{k'} \mathrm{Ad}(\rho'))=
k'\otimes_k\HH^2(G,\frac{\varpi^\ell W}{\varpi^{\ell+1} W}\otimes_k \mathrm{Ad}(\rho))$$
which is obtained from $c$ by base changing from $k$ to $k'$. (Note that we use here that $W'$ is unramified over $W$.)  Since by the arguments in the proof of \cite[Thm. 1.3]{malle}, the class $c'$ vanishes, this means that $c$ must vanish
as well, leading to a lift of $G_\rho$ to a subgroup $G_{\rho,\ell+1}$ of $\mathrm{GL}_n(W/\varpi^{\ell+1} W)$ such that the injection 
$\mathrm{GL}_n(W/\varpi^{\ell+1} W)\to \mathrm{GL}_n(W'/\varpi^{\ell+1} W')$ provides a subgroup $G_{\rho',\ell+1}$ of
$\mathrm{GL}_n(W'/\varpi^{\ell+1} W')$ lifting $G_{\rho'}$. Taking the limit  as $\ell \to \infty$ of these lifts completes the proof of Lemma \ref{lem:improvemalle}.
\end{proof}

Using Lemma \ref{lem:improvemalle} together with \cite[Lemma 2.2.2]{bleherTAMS2009}, we can now prove Theorem \ref{thm:supermain}.

\bigskip

\noindent
\textit{Proof of Theorem $\ref{thm:supermain}$.}
By Lemma \ref{lem:improvemalle},
$V$ can be lifted to an endo-trivial $WG$-module $V_W$. 
Define  $M=V_W\otimes_W WG$. Then $M$ is a $WG$-$WG$-bimodule, where the $G$-action on the left is diagonal and the $G$-action on the right is on the $WG$-factor.
Hence $M$ is projective (in fact, free) both as a left and as a right $WG$-module. 
Define $N=\mathrm{Hom}_{W}(M,W)$. Then
\begin{eqnarray*}
M\otimes_{WG} N  &\cong& WG\oplus P\\
N\otimes_{WG} M &\cong& WG\oplus Q
\end{eqnarray*}
where $P$ and $Q$ are projective $WG$-$WG$-bimodules, and both of these isomorphisms are isomorphisms of $WG$-$WG$-bimodules. 
In particular, $M$ and $N$ define
a stable autoequivalence of Morita type of the $W$-stable category $WG$-\underline{mod}, which is the quotient category of the category $WG$-mod of finitely generated
$WG$-modules by the subcategory of relatively $W$-projective modules (see, e.g., \cite{linckelstable} for more details). 
If $V_0=k$ is the trivial simple $kG$-module then 
\begin{eqnarray*}
(k\otimes_WM)\otimes_{kG} V_0 &\cong&
(k\otimes_W V_W)\otimes_k(kG\otimes_{kG}V_0)\\
&\cong& V\otimes_k V_0\;\cong\;V.
\end{eqnarray*}
Hence it follows by \cite[Lemma 2.2.2]{bleherTAMS2009} that  $R(G,V)\cong R(G,V_0)$. By \cite[Sect. 1.4]{maz1}, $R(G,V_0)$ is isomorphic to $WG^{\mathrm{ab},p}$
and the universal deformation of $V_0$ is given by the isomorphism class of $U(G,V_0)=WG^{\mathrm{ab},p}$ as an $R(G,V_0)G$-module on which 
$g\in G$ acts as multiplication by its image $\overline{g}\in G^{\mathrm{ab},p}$.
The proof of \cite[Lemma 2.2.2]{bleherTAMS2009} shows that the isomorphism class of the universal deformation of $V$ is given by the isomorphism class of
\begin{eqnarray*}
(M\otimes_W R(G,V_0))\otimes_{R(G,V_0)G}U(G,V_0)&\cong& (V_W\otimes_W R(G,V_0)G)\otimes_{R(G,V_0)G}U(G,V_0)\\
&\cong& V_W\otimes_WU(G,V_0)\\
&\cong& V_W\otimes_WWG^{\mathrm{ab},p}
\end{eqnarray*}
where $g\in G$ acts diagonally as multiplication by $g$ on $V_W$ and as multiplication by its image $\overline{g}$ on $WG^{\mathrm{ab},p}$.
This proves parts (i) and (ii) of Theorem \ref{thm:supermain}.

For part (iii) of Theorem \ref{thm:supermain}, let $V_1$ be a non-projective indecomposable direct summand of $V$. Then $V_1$ is unique
up to isomorphism, and $V$ is the direct sum of $V_1$ and a projective $kG$-module. In particular, we have $R(G,V)\cong R(G,V_1)$. 
Since the $k$-dimension of $V_1$ is not divisible by $p$, it follows by \cite[Thm. (19.26)]{CR} that the vertices of $V_1$ are Sylow $p$-subgroups of $G$.
Therefore, $D$ must also be a Sylow $p$-subgroup. Since
$G^{\mathrm{ab},p}$ is a subquotient group of a Sylow $p$-subgroup of $G$, 
it follows that $R(G,V)$ is 
isomorphic to a subquotient ring of $WD$.
\hfill $\Box$

\bigskip

We now turn to nilpotent blocks, and the proof of Corollary \ref{cor:nilpotent}. The following remark recalls the most important definitions and results for nilpotent blocks.

\begin{bigrem}
\label{rem:nilpotent}
As in \cite{puig}, we assume  that $k$ is algebraically closed.
Let $D$ be a finite $p$-group, let $G$ be a finite group and let $B$ be a nilpotent block of $kG$ that has $D$ as a defect group.
By \cite{brouepuig}, this means that whenever $(D_1,e_1)$ is a $B$-Brauer pair then 
the quotient $N_G(D_1,e_1)/C_G(D_1)$ is a $p$-group. In other words,
for all subgroups $D_1$ of $D$ and for all block idempotents $e_1$ of $kC_G(D_1)$ associated with $B$, the 
quotient of the stabilizer $N_G(D_1,e_1)$ of $e_1$ in $N_G(D_1)$ by the centralizer $C_G(D_1)$
is a $p$-group. In \cite{puig}, Puig rephrased this definition using the theory of local pointed groups. 

Let $\hat{B}$ be the block of $WG$ corresponding to $B$. Then $\hat{B}$ is also nilpotent, and
by \cite[\S1.4]{puig}, $\hat{B}$ is Morita equivalent to $WD$.  In
\cite[Thm. 8.2]{puig2}, Puig showed that the converse is also true in a very strong way.
Namely, if $\hat{B'}$ is another  block over $W$ such that there is a stable equivalence of
Morita type between $\hat{B}$ and $\hat{B'}$, then $\hat{B'}$ is also nilpotent. Hence Corollary
\ref{cor:nilpotent} can be applied in particular if there is only known to be a stable
equivalence of Morita type between $\hat{B}$ and $WD$.
\end{bigrem}

\bigskip

\noindent
\textit{Proof of Corollary $\ref{cor:nilpotent}$.}
Let $D$ be a defect group of $B$, so that $B$ is a nilpotent block of $kG$. By Remark \ref{rem:nilpotent}, this means that
the corresponding block $\hat{B}$  of $WG$ is also nilpotent and Morita equivalent to $WD$. 
Suppose $V$ is a finitely generated $B$-module, and $V'$ is the $kD$-module corresponding to
$V$ under this Morita equivalence.
Then the stable endomorphism ring of $V$ is isomorphic to $k$ if and only if the
stable endomorphism ring of $V'$ is isomorphic to $k$. 
Moreover, it follows for example from \cite[Prop. 2.5]{bl} that $R(G,V)\cong R(D,V')$.
By Theorem \ref{thm:supermain} and \cite[Thm. 1.1]{diloc},
this implies that $R(G,V)\cong WG^{\mathrm{ab},p}$.
\hspace*{\fill} $\Box$

\bigskip

The second goal of this paper is to
give an explicit description of the universal deformations of all endo-trivial $kG$-modules 
$V$ in the case when $G$ is a $p$-group and the group of equivalence classes of endo-trivial modules 
$T(G)$ has a non-trivial torsion subgroup. By \cite[Thm. 1.1]{carl1.5}, 
this means that either $G$ is cyclic of order at least $3$, or $p = 2$ and $G = \D$ is a semidihedral or 
generalized quaternion $2$-group. We first treat the case of cyclic $p$-groups in the following example
(the other cases will be done in Sections \ref{s:semidih} - \ref{s:quat}).

\begin{bigexample}
\label{ex:cyclic}
Let $G=\langle \sigma\rangle$ be a cyclic $p$-group of order $p^d\ge 3$.
By \cite[Cor. 8.8]{dade}, the indecomposable endo-trivial $kG$-modules are isomorphic to $\Omega^n(k)$ for some positive integer $n$.
Since $\Omega^2(k)\cong k$, the only indecomposable endo-trivial $kG$-modules are, up to isomorphism, given by the trivial simple $kG$-module $V_0=k$ and its first syzygy
$V_1=\Omega(V_0)$. Note that the stable Auslander-Reiten quiver of $kG$ is a finite one-tube with $p^d-1$ vertices and that $V_0$ and $V_1$ lie at the two ends of this one-tube.

By \cite[Sect. 1.4]{maz1}, it follows that $R(G,V_0)$ is isomorphic to $WG$
and the universal deformation of $V_0$ is given by the isomorphism class of $U(G,V_0)=WG$ which is viewed as an $R(G,V_0)G$-module on which 
$\sigma$ acts as multiplication by $\sigma$. On the other hand, if $V_1=\Omega(V_0)$, we can use Remark \ref{rem:syzygylift} to construct a lift of $V_1$ over $W$.
More explicitly, a lift of $V_1$ over $W$ is given by
the augmentation ideal $U_1$ of $WG$. In particular, $U_1 = WG(1-\sigma)$. By Theorem \ref{thm:supermain},
$R(G,V_1)\cong WG$ and the universal deformation of $V_1$ is given by the isomorphism class of 
$$U(G,V_1)=U_1\otimes_WWG=WG(1-\sigma) \otimes_W WG$$
on which $\sigma$ acts diagonally. 
In particular, a representation of $U(G,V_1)$ is given by
\begin{eqnarray*}
\rho_{U(G,V_1)}:\qquad G &\longrightarrow &\mathrm{GL}_{p^d-1}(WG)\\
\sigma&\mapsto& \left(\begin{array}{ccccc}
0&0&\cdots&0&-\sigma\\ 
\sigma&0&\cdots&0&-\sigma\\
0&\ddots&\ddots&\vdots&\vdots\\
\vdots&\ddots&\sigma&0&-\sigma\\
0&\cdots&0&\sigma&-\sigma\end{array}\right).
\end{eqnarray*}
\end{bigexample}


\section{The semi-dihedral  $2$-groups and their endo-trivial modules}
\label{s:semidih}
\setcounter{equation}{0}

Fix an integer $d\ge 4$, and let $\SD$ be a semidihedral group of order $2^d$ with the following presentation
\begin{equation}
\label{eq:semid}
\SD=\langle x,y\;|\; x^{2^{d-1}}=y^2=1, yxy^{-1}= x^{2^{d-2}-1}\rangle\,.
\end{equation}
Then $\SD=\langle yx,y\rangle$ where $yx$ has order 4 and $y$ has order 2. 
Define $z=x^{2^{d-2}}$ to be the unique non-trivial central element of $\SD$, which is also
the unique involution of $\SD$. Define
\begin{equation}
\label{eq:SDab2}
\overline{\SD}=\SD^{\mathrm{ab},2} = \langle \overline{yx},\overline{y}\rangle = \{\overline{1},\overline{x},\overline{y},\overline{yx}\}
\end{equation}
where for $g\in\SD$, $\overline{g}$ denotes its image in $\overline{\SD}$.

Let $k$ be a field of characteristic 2 and let $W$ be a complete discrete valuation ring of characteristic 0 that has $k$ as its residue field.

We first describe the endo-trivial $k\SD$-modules using \cite[Sect. 7]{CT}.

\begin{bigrem}
\label{rem:endotrivSD}
Define
\begin{equation}
\label{eq:HE}
H=\langle x^{2^{d-3}}, yx\rangle\quad\mbox{and}\quad E=\langle y,z\rangle.
\end{equation}
Then $E$ is the only elementary abelian subgroup of rank 2 of $\SD$ up to conjugacy,
and $H$ represents the unique conjugacy class of quaternion subgroups of order 8 of $\SD$.
Let $T(\SD)$ denote the group of equivalence classes of 
endo-trivial $k\SD$-modules. 
Consider the restriction map
$$\Xi_{\SD}:T(\SD)\to T(E)\times T(H)\cong \mathbb{Z}\times \mathbb{Z}/4.$$
By the proof of \cite[Thm. 7.1]{CT}, $\Xi_{\SD}$ is injective and the image of $\Xi_{\SD}$ is generated by $(\Omega_E,\Omega_H)$ and 
$(-\Omega_E,\Omega_H)$. 
Define
\begin{equation}
\label{eq:YL}
Y=k\;\SD/\langle y\rangle
\qquad\mbox{and}\qquad L=\mathrm{rad}(Y)
\end{equation}
where $\SD/\langle y \rangle$ denotes the set of distinct left cosets of 
$\langle y\rangle$ in $\SD$ and $Y$ is the corresponding permutation module for $\SD$
over $k$. It follows from \cite[Sect. 7]{CT} that
\begin{equation}
\label{eq:L!SD}
\mathrm{Res}^{\SD}_H\, L \cong \Omega^1_H(k)\oplus \mathrm{(free)}\quad\mbox{and}\quad
\mathrm{Res}^{\SD}_E\, L \cong \Omega^{-1}_E(k)\oplus \mathrm{(free)}.
\end{equation}
Hence $T(\SD)$ is isomorphic to $\mathbb{Z}\oplus \mathbb{Z}/2$, and it is generated by $[\Omega^1_{\SD}(k)]$ and
$[\Omega^1_{\SD}(L)]$. Since $[\Omega^1_{\SD}(L)]$ has order 2, it follows that each element of $T(\SD)$ is of the form 
$[\Omega^i_{\SD}(k)]$ or $[\Omega^i_{\SD}(L)]$ for some $i\in\mathbb{Z}$.
\end{bigrem}

The following remark gives a different description of the $k\SD$-module $L$ in 
(\ref{eq:YL}) to highlight the similarity between the endo-trivial modules for the 
semidihedral 2-groups and the generalized quaternion 2-groups (see Remark
\ref{rem:endotrivQ}).

\begin{bigrem}
\label{rem:vary}
For all $i\in\mathbb{Z}$, define $b_i\in L$ to be the element
$b_i = x^i\langle y\rangle - x^{i-1}\langle y\rangle$.
Then $\{b_1,\ldots, b_{2^{d-1}-1}\}$ is a $k$-basis for $L$, and we have the relations
\begin{equation}
\label{eq:rel}
b_{2^{d-1}} = -\sum_{j=1}^{2^{d-1}-1} b_j\qquad \mbox{and}
\quad b_i=b_{i+2^{d-1}}\quad\mbox{for all $i$}.
\end{equation}
Let $X=\langle x\rangle$ be the unique cyclic subgroup of $\SD$ of order $2^{d-1}$, let
$\Tr_X=\sum_{j=0}^{2^{d-1}-1}x^j$ be the trace element of $X$, and let $\overline{S}$ 
be the ring $kX/(k\cdot \Tr_X)$. For all $i\in\mathbb{Z}$, define $c_i\in \overline{S}$
to be the image of $x^{i-1}$ in $\overline{S}$.
Then $\{c_1,\ldots, c_{2^{d-1}-1}\}$ is a $k$-basis for $\overline{S}$, and
the $c_i$ satisfy the same relations as the $b_i$ in (\ref{eq:rel}).
Thus, there is a $kX$-module isomorphism
$$f:\quad L \to \overline{S},\qquad f(b_i)=c_i\quad\mbox{for }1\le i\le 2^{d-1}-1.$$
Using the action of $y$ on $L$, we can then use $f$ to define a compatible action
of $y$ on $\overline{S}$ as follows:
\begin{equation}
\label{eq:yacts}
y\cdot c_i \quad =\quad \sum_{j=(i-1)(2^{d-2}-1)+1}^{i(2^{d-2}-1)} c_j.
\end{equation}
Let $S=WX/ (W\cdot \Tr_X)$, and
let $\sigma_x$ denote the $W$-linear automorphism of $S$ given by multiplication by $x$.
Finding a lift of $L$ over $W$ is then equivalent to
constructing a $W$-linear automorphism $\sigma_y$ of $S$ such that $\sigma_x$ and $\sigma_y$ satisfy
the relations of $x$ and $y$ in $\D$ and such that  $\sigma_y$ reduces 
(modulo the maximal ideal of $W$) to the action of 
$y$ on $\overline{S}$ given by (\ref{eq:yacts}).
Let $L_W$ be the kernel of the $W\SD$-module homomorphism
$\pi_W:W\;\SD/\langle y\rangle\to W$, defined by $\pi_W(x^i\langle y\rangle) = 1_W$ 
for all $0\le i \le 2^{d-1}-1$. Then $L_W$ defines a lift of $L$ over $W$, and we can
lift the $k$-basis $\{b_1,\ldots,b_{2^{d-1}-1}\}$ of $L$ to a corresponding
$W$-basis $\{\hat{b}_1,\ldots,\hat{b}_{2^{d-1}-1}\}$ of $L_W$ that satisfies the same
relations as in (\ref{eq:rel}). This means that we can also lift the $k$-basis 
$\{c_1,\ldots,c_{2^{d-1}-1}\}$ of $\overline{S}$ to a corresponding
$W$-basis $\{\hat{c}_1,\ldots,\hat{c}_{2^{d-1}-1}\}$ of $S$ and define
$\sigma_y(\hat{c}_i)=\sum_{j=(i-1)(2^{d-2}-1)+1}^{i(2^{d-2}-1)} \hat{c}_j$
to obtain that $S$ defines a lift of the $k\SD$-module $\overline{S}$ over $W$.
\end{bigrem}

By Theorem \ref{thm:supermain}, we know that the universal deformation ring of every endo-trivial $k\SD$-module is isomorphic to
$W\overline{\SD}$ where $\overline{\SD}$ is as in (\ref{eq:SDab2}). The following result gives an explicit description of the universal
deformation of every endo-trivial $k\SD$-module $V$. Since projective $k\SD$-modules can always be lifted over $W$ (see, for example, 
\cite[Props. (6.5) and (6.7)]{CR}), we can concentrate on the indecomposable endo-trivial $k\SD$-modules. 
Using Remark \ref{rem:endotrivSD}, the following result is a consequence of Theorem \ref{thm:supermain} and Remark \ref{rem:syzygylift}.

\begin{bigprop}
\label{prop:unidefSD}
Let $R=W\overline{\SD}$. 
If $V_0=k$ is the trivial simple $k\SD$-module, let $V_{0,W}=W$ with trivial $\SD$-action.
If $V_1=L$ is as in $(\ref{eq:YL})$, let $V_{1,W}=\mathrm{Ker}(\pi_W)$ where 
$\pi_W:W\;\SD/\langle y\rangle\to W$
is the $W\SD$-module homomorphism defined by $\pi_W(x^a\langle y\rangle) = 1_W$ for all $0\le a \le 2^{d-1}-1$. 

Let $i\in\mathbb{Z}$ and $j\in\{0,1\}$. Then $\Omega^i_{W\SD}(V_{j,W})$, as defined in Remark 
$\ref{rem:syzygylift}$, is a lift of $\Omega_{\SD}^i(V_j)$ over $W$. Moreover, the universal deformation of $\Omega_{\SD}^i(V_j)$ is given by the isomorphism class of 
the $R\,\SD$-module $U(\SD,\Omega_{\SD}^i(V_j))=\Omega^i_{W\SD}(V_{j,W})\otimes_W R$ on which $x$ $($resp. $y$$)$ acts diagonally as multiplication by 
$x\otimes\overline{x}$ $($resp. $y\otimes\overline{y}$$)$.
\end{bigprop}


\section{The quaternion  group of order $8$ and its endo-trivial modules}
\label{s:quat8}
\setcounter{equation}{0}

Suppose $Q_8$ is a quaternion group of order 8 with the following presentation
\begin{equation}
\label{eq:quat8}
Q_8=\langle x,y\;|\; x^4=1, x^2=y^2=(yx)^2\rangle.
\end{equation}
Then $Q_8=\langle yx,y\rangle$ where $yx$ and $y$ both have order 4. Define
\begin{equation}
\label{eq:Q8ab2}
\overline{Q}_8=Q_8^{\mathrm{ab},2} = \langle \overline{yx},\overline{y}\rangle = \{\overline{1},\overline{x},\overline{y},\overline{yx}\}
\end{equation}
where for $g\in Q_8$, $\overline{g}$ denotes its image in $\overline{Q}_8$.

Let $k$ be a field of characteristic 2 and let $W$ be a complete discrete valuation ring of characteristic 0 that has $k$ as its residue field.

We first describe the endo-trivial $kQ_8$-modules using \cite[Thm. 6.3]{CT}.

\begin{bigrem}
\label{rem:endotrivQ8}
Let $T(Q_8)$ denote the group of equivalence classes of endo-trivial $kQ_8$-modules. 
\begin{enumerate}
\item[(a)] If $k$ does not contain a primitive cube root of unity, then $T(Q_8)$ is a cyclic group of order 4 that is generated by 
$[\Omega^1_{Q_8}(k)]$.
\item[(b)] If $k$ contains a primitive cube root $\omega$ of unity, then $T(Q_8)$ is isomorphic to $\mathbb{Z}/4\oplus \mathbb{Z}/2$. 
More precisely, let $L$ be the $3$-dimensional $kQ_8$-module with a
representation $\rho_L:Q_8 \to\mathrm{GL}_{3}(k)$ given by
\begin{equation}
\label{eq:rhoL}
\rho_L(x)= \left(\begin{array}{ccc}
1&0&0\\1&1&0\\0&1&1
\end{array}\right),\qquad \rho_L(y)= \left(\begin{array}{ccc}
1&0&0\\\omega&1&0\\0&\omega^2&1
\end{array}\right).
\end{equation}
Then $T(Q_8)$ is generated by $[\Omega^1_{Q_8}(k)]$ and $[\Omega^1_{Q_8}(L)]$, and $[\Omega^1_{Q_8}(L)]$ has order $2$.
In particular, each element of $T(Q_8)$ is of the form $[\Omega^i_{Q_8}(k)]$ or $[\Omega^i_{Q_8}(L)]$ for some $i\in\{0,1,2,3\}$.
\end{enumerate}
\end{bigrem}

The description of $L$ in Remark \ref{rem:endotrivQ8}(b) shows that 
the restriction of $L$ to the maximal cyclic subgroup $X=\langle x\rangle$ of  $Q_8$ 
is isomorphic to $\overline{S}=kX/(k\cdot \Tr_X)$, where $\Tr_X$ is the sum of the 
elements in $X$. Using this isomorphism, we can then explicitly find the $k$-linear
automorphism of $\overline{S}$ corresponding to the action of $y$ on $L$. This
defines a $kQ_8$-module structure on $\overline{S}$ such that
$\overline{S}$ is isomorphic to $L$ as a $kQ_8$-module. As in Remark \ref{rem:vary},
this highlights the similarity between the endo-trivial modules for $Q_8$,
when $k$ contains a primitive cube root of unity, and
the generalized quaternion groups of higher 2-power order (see Remark
\ref{rem:endotrivQ}).

By Theorem \ref{thm:supermain}, we know that the universal deformation ring of every endo-trivial $kQ_8$-module is isomorphic to
$W\overline{Q}_8$ where $\overline{Q}_8$ is as in (\ref{eq:Qab2}). We now give an explicit description of the universal
deformation of every endo-trivial $kQ_8$-module $V$. As in Section \ref{s:semidih}, we can concentrate on the indecomposable 
endo-trivial $kQ_8$-modules.

\begin{bigprop}
\label{prop:unidefQ8}
Let $R=W\overline{Q}_8$. If $V_0=k$ is the trivial simple $kQ_8$-module, let $V_{0,W}=W$ with trivial $Q_8$-action.
\begin{enumerate}
\item[(i)] If $k$ contains a primitive cube root of unity $\omega$, let $\hat{\omega}$ be a primitive cube root of unity in $W$, and
let $V_1=L$ be as in $(\ref{eq:rhoL})$. Then the representation $\rho_{L,W}:Q_8 \to\mathrm{GL}_{3}(W)$ given by
\begin{equation}
\label{eq:rhoLW}
\rho_{L,W}(x) =  \left(\begin{array}{rrr} 1&0&0\\1&-1&2\\0&-1&1\end{array}\right),\qquad
\rho_{L,W}(y) = \left(\begin{array}{rrr} 1&0\;&0\;\\-\hat{\omega}&-1\;&-2\hat{\omega}\\0&\hat{\omega}^2&1\;\\\end{array}\right)
\end{equation}
defines a $WQ_8$-module $V_{1,W}=L_W$ that is a lift of $L=V_1$ over $W$.
\item[(ii)]
Let $i\in\{0,1,2,3\}$ and let $j\in\{0,1\}$ $($resp. $j=0$$)$ if $k$ contains $($resp. does not contain$)$ a primitive cube root of unity $\omega$. Then 
$\Omega^i_{WQ_8}(V_{j,W})$, as defined in Remark 
$\ref{rem:syzygylift}$, is a lift of $\Omega_{Q_8}^i(V_j)$ over $W$. Moreover, the universal deformation of $\Omega_{Q_8}^i(V_j)$ is given by the isomorphism class of 
the $R\,Q_8$-module $U(Q_8,\Omega_{Q_8}^i(V_j))=\Omega^i_{WQ_8}(V_{j,W})\otimes_W R$ on which $x$ $($resp. $y$$)$ acts diagonally as multiplication by 
$x\otimes\overline{x}$ $($resp. $y\otimes\overline{y}$$)$.
\end{enumerate}
\end{bigprop}

\begin{proof}
Since the reduction modulo $m_W$ of the matrices given in $(\ref{eq:rhoLW})$ gives the matrices given in $(\ref{eq:rhoL})$ and
since they satisfy the relations of $x$ and $y$ in $Q_8$, it follows that $V_{1,W}=L_W$ defines a lift of $L=V_1$ over $W$.
Part (ii) of Proposition \ref{prop:unidefQ8} follows now  from Theorem \ref{thm:supermain} and Remark \ref{rem:syzygylift}.
\end{proof}


\section{The generalized quaternion  $2$-groups and their endo-trivial modules}
\label{s:quat}
\setcounter{equation}{0}

Fix an integer $d\ge 4$, and let $\Q$ be a generalized quaternion group of order $2^d$ with the following presentation
\begin{equation}
\label{eq:quat}
\Q=\langle x,y\;|\; x^{2^{d-1}}=1, x^{2^{d-2}}=y^2, yxy^{-1}= x^{-1}\rangle.
\end{equation}
Then $\Q=\langle yx,y\rangle$ where $yx$ and $y$ both have order 4.
Define $z=x^{2^{d-2}}$ to be the unique non-trivial central element of $\Q$, which is also
the unique involution of $\Q$. Define
\begin{equation}
\label{eq:Qab2}
\overline{\Q}=\Q^{\mathrm{ab},2} = \langle \overline{yx},\overline{y}\rangle = \{\overline{1},\overline{x},\overline{y},\overline{yx}\}
\end{equation}
where for $g\in\Q$, $\overline{g}$ denotes its image in $\overline{\Q}$.

Let $k$ be a field of characteristic 2 and let $W$ be a complete discrete valuation ring of characteristic 0 that has $k$ as its residue field.
Let $\mathbb{F}_2$ be the prime subfield of $k$ with two elements,
and let $\mathbb{Z}_2$ be the 2-adic integers such that $\mathbb{Z}_2\subseteq W$.

The following remark summarizes the description of the endo-trivial $k\Q$-modules from \cite[Sect. 6]{CT}.

\begin{bigrem}
\label{rem:endotrivQ}
Let $T(\Q)$ denote the group of equivalence classes of endo-trivial $k\Q$-modules. 
Then $T(\Q)$ is isomorphic to $\mathbb{Z}/4\oplus \mathbb{Z}/2$. More precisely, the group $T(\Q)$ is constructed as follows.

\begin{enumerate}
\item[(a)] Let $X=\langle x\rangle$ be the unique cyclic subgroup of $\Q$ of order $2^{d-1}$, let
$\Tr_X=\sum_{j=0}^{2^{d-1}-1}x^j$ be the trace element of $X$, and let $\overline{S}$ be the ring
$kX/(k\cdot \Tr_X)$. 
Let $*$ denote the involution of $\overline{S}$ that is induced by inversion on $X$.
In the proof of \cite[Lemma 6.4]{CT} 
an element $\gamma\in \mathbb{F}_2X$ is found that satisfies
\begin{equation}
\label{eq:betamod2}
\gamma\,\gamma^*= x^{2^{d-2}} \mod (\mathbb{F}_2\cdot \Tr_X)\,.
\end{equation}
Moreover, it is shown that every element $\gamma\in \mathbb{F}_2 X$ satisfying $(\ref{eq:betamod2})$
defines a $k\Q$-module structure on $\overline{S}=kX/(k\cdot \Tr_X)$.
Namely, the action of $y$ on $s\in\overline{S}$ is given by
$y\cdot s=\gamma\,s^*$.
Since $\overline{S}\cong \Omega^1_{\langle x\rangle}(k)$ as $kX$-module, 
this means that any such $k\Q$-module is an endo-trivial $k\Q$-module of $k$-dimension $2^{d-1}-1$.

\item[(b)] Let $L$ be an endo-trivial $k\Q$-module of $k$-dimension $2^{d-1}-1$ such that
$\mathrm{Res}^{\Q}_{\langle x\rangle}\,L\cong\Omega^1_{\langle x\rangle}(k)$.
There are two conjugacy classes of quaternion subgroups of $\Q$ of order 8, represented by
\begin{equation}
\label{eq:twoQ8}
H=\langle yx, x^{2^{d-3}} \rangle\quad\mbox{and}\quad H'=\langle y, x^{2^{d-3}} \rangle.
\end{equation}
Consider the restriction map
$$\Xi_{\Q}:T(\Q)\to T(H)\times T(H')\cong \mathbb{Z}/4\times \mathbb{Z}/4.$$
By the proof of \cite[Thm. 6.5]{CT}, $\Xi_{\Q}$ is injective and the image of $\Xi_{\Q}$ is generated by 
$(\Omega_H,\Omega_{H'})$ and $(\Omega_H,-\Omega_{H'})$. 
Moreover, $\Xi_{\Q}([L])=\pm (\Omega_H,-\Omega_{H'})$. In other words, there exist $\epsilon_L\in\{\pm 1\}$ such that
\begin{equation}
\label{eq:L!Q}
\mathrm{Res}^{\Q}_H\, L \cong \Omega^{\epsilon_L}_H(k)\oplus \mathrm{(free)}\quad\mbox{and}\quad
\mathrm{Res}^{\Q}_{H'}\, L \cong \Omega^{-\epsilon_L}_{H'}(k)\oplus \mathrm{(free)}.
\end{equation}
Hence $T(\Q)$ is generated by $[\Omega^1_{\Q}(k)]$ and $[\Omega^1_{\Q}(L)]$, and $[\Omega^1_{\Q}(L)]$ has order $2$.
In particular, each element of $T(\Q)$ is of the form $[\Omega^i_{\Q}(k)]$ or $[\Omega^i_{\Q}(L)]$ for some $i\in\{0,1,2,3\}$.
\end{enumerate}
\end{bigrem}

By Theorem \ref{thm:supermain}, we know that the universal deformation ring of every endo-trivial $k\Q$-module is isomorphic to
$W\,\overline{\Q}$ where $\overline{\Q}$ is as in (\ref{eq:Qab2}). We now give an explicit description of the universal
deformation of every endo-trivial $k\Q$-module $V$. As in Sections \ref{s:semidih} and \ref{s:quat8}, we can concentrate on the indecomposable 
endo-trivial $k\Q$-modules. 

We first construct an explicit lift over $W$ of any endo-trivial $k\Q$-module of $k$-dimension $2^{d-1}-1$ as 
defined in Remark \ref{rem:endotrivQ}(a). The crucial step is to 
lift any element $\gamma\in\mathbb{F}_2X$ satisfying $(\ref{eq:betamod2})$ to a corresponding
element $\beta\in\mathbb{Z}_2X$. We need two propositions.

\begin{bigprop}
\label{prop:betaliftpowerseries}
For $j\ge 0$, define $p_j\in \mathbb{Z}[t]$ inductively by $p_0=2$, $p_1=t$ and the recurrence relation
\begin{equation}
\label{eq:pj}
p_{j+1} = tp_j-p_{j-1} \qquad \mbox{for $j\ge1$}.
\end{equation}
Let $\tau=\sum_{j=0}^{2^{d-2}-1} p_j$ and let $a = 1+2^{d-4}\,t$.
Then the equation
\begin{equation}
\label{eq:crucial}
b^2 + tab + a^2 = 1-\tau
\end{equation}
has a solution $b$ in $\mathbb{Z}_2[[t]]$. Moreover, the discriminant $\Delta(t)$ of $(\ref{eq:crucial})$ is a polynomial in $\mathbb{Z}[t]$
such that  $\Delta(t) = t^2 ( (1-t)^2-8)$ if $d=4$, and if $d\ge 5$ then
$\Delta(t)=t^2(1+4 \,m(t) )$ for some polynomial $m(t)\in\mathbb{Z}[t]$ whose constant coefficient is divisible by $2$.
\end{bigprop}

\begin{proof}
Consider the recurrence relation for $p_j$ in (\ref{eq:pj}). The associated quadratic equation
\begin{equation}
\label{eq:niceq}
Y^2 - t Y+ 1 = 0
\end{equation}
in the variable $Y$
has two solutions $y_1, y_2$ which are integral units over the ring $\mathbb{Z}_2[[t]]$. Note that $y_1^0+y_2^0=2=p_0$ and
$y_1^1+y_2^1=t=p_1$.
Since $y_1$ and $y_2$ are roots of (\ref{eq:niceq}), it follows that the recurrence relation (\ref{eq:pj}) holds for $y_1^j + y_2^j$,
which implies that
\begin{equation}
\label{eq:pjbetter}
p_j = y_1^j + y_2^j
\end{equation}
for all $j\ge 0$.
Using that $(1-y_1)(1-y_2)=2-t$ and that $y_1,y_2$ satisfy (\ref{eq:niceq}), we rewrite $\tau$ as
\begin{eqnarray}
\tau &=& \sum_{j = 0}^{2^{d-2} - 1} ( y_1^j + y_2^j )\; = \;  \frac{y_1^{2^{d-2}} - 1}{y_1 - 1} + \frac{y_2^{2^{d-2}} - 1}{y_2 - 1} \nonumber\\
&=& \frac{(y_2 -1) \left ( \sum_{a = 1}^{2^{d-3}} \genfrac(){0pt}{1}{2^{d-3}}{a} 
(ty_1)^a (-1)^a \right )  + (y_1 -1) \left ( \sum_{a= 1}^{2^{d-3}} \genfrac(){0pt}{1}{2^{d-3}}{a} (ty_2)^a (-1)^a \right )}{2 - t}\nonumber\\
&=& \frac{t \, g(t)}{2 - t}\nonumber
\end{eqnarray}
for some $g(t) \in \mathbb{Z}_2[[t]]$.  Since $\tau \in \mathbb{Z}_2[[t]]$ and $\mathbb{Z}_2[[t]]$ is a unique factorization domain, 
there exists $h(t) \in \mathbb{Z}_2[[t]]$ such that $g(t) = (2-t)\, h(t)$.  
Let $A=\mathbb{Z}_2[[t]][y_1,y_2]$. Using that $y_1,y_2$ satisfy (\ref{eq:niceq}), we have $t^2 y_1^2 \equiv - t^2\mod t^3 A$ and 
$t^2 y_2^2 \equiv -t^2\mod t^3 A$. Hence we obtain the congruence relation
\begin{eqnarray}
t(2-t)h(t)  &\equiv &  (y_2 -1) \left ( \sum_{a = 1}^{2} \genfrac(){0pt}{0}{2^{d-3}}{a} (ty_1)^a (-1)^a \right )  + (y_1 -1) \left ( \sum_{a = 1}^{2} \genfrac(){0pt}{0}{2^{d-3}}{a} (ty_2)^a (-1)^a \right ) 
\nonumber \\
&\equiv&   (y_2 -1) \left (  2^{d-3} (-ty_1)  - \genfrac(){0pt}{0}{2^{d-3}}{2} t^2 \right ) + (y_1 -1) \left (  2^{d-3}  (-ty_2) - \genfrac(){0pt}{0}{2^{d-3}}{2} t^2 \right ) 
\nonumber \\
&\equiv &  t \,2^{d-3} \left ( (y_2 -1)   (-y_1)  +  (y_1 -1)  (-y_2)  \right ) + 
t^2  \genfrac(){0pt}{0}{2^{d-3}}{2} \left ( - (y_2 -1) - (y_1 -1) \right )  
\nonumber \\
&\equiv &  t \,2^{d-3} \left ( -2 + t  \right ) + 
t^2  \genfrac(){0pt}{0}{2^{d-3}}{2}  (2 - t) 
\nonumber \\
&\equiv &   t \,(2-t) \left ( -2^{d-3} + t  \genfrac(){0pt}{0}{2^{d-3}}{2}   \right )  \mod t^3 A.
\nonumber 
\end{eqnarray}
This implies
\begin{equation}
\label{eq:tauprecise}
\tau\equiv t \left ( -2^{d-3} + t  \genfrac(){0pt}{0}{2^{d-3}}{2}   \right )   \quad \mathrm{mod}\quad t^3\, \mathbb{Z}_2[[t]]\,.
\end{equation}
Setting $a = 1+2^{d-4}t$ in the quadratic equation (\ref{eq:crucial})
gives the discriminant
\begin{eqnarray}
\label{eq:Delta}
\Delta(t)&=&t^2(1+2^{d-4}t)^2 - 4((1+2^{d-4}t)^2 - (1-\tau))  \\
&=&
t^2(1+2^{d-4}t)^2 - 4(2^{2d-8}t^2 +2^{d-3}t+ \tau)\nonumber
\end{eqnarray}
which is a polynomial in $\mathbb{Z}[t]$, since $p_j$, $j\ge 0$, and hence also $\tau$, are polynomials in $\mathbb{Z}[t]$.

If $d\ge 5$ then 
we obtain from (\ref{eq:tauprecise}) the congruence relation
\begin{eqnarray}
\Delta(t)&\equiv& 
t^2- 4\left( 2^{d-3}t+2^{2d-8}t^2 +t \left( -2^{d-3} + t  \genfrac(){0pt}{0}{2^{d-3}}{2} \right)\right) \nonumber\\
& \equiv& t^2 \left(1- 4\left(2^{2d-8}+2^{d-4}\left(2^{d-3} -1\right)\right)\right) \mod 4 t^3\, \mathbb{Z}_2[[t]]\,.\nonumber
\end{eqnarray}
If $d=4$ then $\tau=\sum_{j=0}^3p_j = t^3+t^2-2t$ and
\begin{eqnarray}
\Delta(t)
&=&t^2(1+2t+t^2) - 4(t^2+2t+t^3+t^2-2t) \nonumber\\
&=& t^2 ( (1-t)^2-8) \nonumber \\
&=& t^2(1-t)^2\left(1-\frac{8}{(1-t)^2}\right)
\quad \mbox{ in } \mathbb{Z}_2[[t]]\,.\nonumber
\end{eqnarray}
In other words, for all $d\ge 4$, there exists a polynomial $f(t)\in\mathbb{Z}_2[t]$ and a
power series $m(t)$ in the maximal ideal $(2,t)\,\mathbb{Z}_2[[t]]$ such that
\begin{equation}
\label{eq:Deltabigd}
\Delta(t)  =f(t)^2 ( 1+ 4\,m(t) )\,.
\end{equation}
Note that if $d\ge 5$ then $m(t)$ is actually a polynomial in $\mathbb{Z}[t]$ whose constant coefficient is divisible by 2.
Using the binomial series
\begin{equation}
\label{eq:binomial}
(1+T)^{1/2} = 1+ \sum_{j=1}^\infty \genfrac(){0pt}{0}{1/2}{j} T^j 
\end{equation}
for $T=4\, m(t)$ and noticing that $\genfrac(){0pt}{1}{1/2}{j} 4^j$ is an element in $2\,\mathbb{Z}_2$ for all $j\ge 1$, it follows that
there exists a power series $\delta(t)\in\mathbb{Z}_2[[t]]$ satisfying
$$(1+2\,\delta(t))^2 = 1+4\,m(t)\,.$$
Hence, if $d\ge 5$ then
\begin{equation}
\label{eq:arrghdbig}
b=-2^{d-5}\,t^2+t\,\delta(t)\qquad (\mbox{resp. } b=-t-2^{d-5}\,t^2-t\,\delta(t))
\end{equation}
and if  $d=4$ then
\begin{equation}
\label{eq:arrghdsmall}
b=-t^2+t(1-t)\,\delta(t)\qquad (\mbox{resp. } b=-t-t(1-t)\,\delta(t))
\end{equation}
are power series in $\mathbb{Z}_2[[t]]$ satisfying $(\ref{eq:crucial})$. 
\end{proof}

\begin{bigprop}
\label{prop:betalift}
Let $X=\langle x\rangle$, let $\Tr_X=\sum_{j=0}^{2^{d-1}-1}x^j$, and let 
$S_2=\mathbb{Z}_2X/(\mathbb{Z}_2\cdot \Tr_X)$. Let $*$ denote the involution of $S_2$ that is induced by inversion on $X$, and let
$S_2^+$ be the subring of $S_2$ that is invariant under $*$. 

The map
\begin{equation}
\label{eq:pi}
\pi:\quad\mathbb{Z}_2[[t]]\to S_2^+\qquad \mbox{given by}\quad \pi(t)=x+x^{-1}
\end{equation}
defines a surjective $\mathbb{Z}_2$-algebra homomorphism. The kernel of $\pi$ is the ideal of $\mathbb{Z}_2[[t]]$ generated
by the distinguished polynomial $\Phi(t)\in\mathbb{Z}[t]$ that is the product of the irreducible monic polynomials $q_j(t)$ in $\mathbb{Z}[t]$
of the numbers $\zeta_{2^j} + \zeta_{2^j}^{-1}$ when $\zeta_{2^j}$ is a root of unity of order $2^j$ and $j = 1,\ldots, d-1$. 

If $\tau,a\in\mathbb{Z}[t]$ and $b\in\mathbb{Z}_2[[t]]$ are as in Proposition $\ref{prop:betaliftpowerseries}$ then
$\beta=\pi(a) + x\, \pi(b)$
satisfies
\begin{equation}
\label{eq:beta}
\beta\beta^*=x^{2^{d-2}}
\end{equation}
in $S_2$.
Moreover, finding $\pi(b)$ explicitly as an element of $S_2^+$ is equivalent to
finding an explicit square root of $\pi(\Delta(t))$ inside $S_2^+\subset S_2$, where 
$\Delta(t)$ is the discriminant of $(\ref{eq:crucial})$. The latter 
is equivalent to taking
square roots of explicitly given elements inside the image of $S_2$ under the
injective $\mathbb{Z}_2$-algebra homomorphism
\begin{equation}
\label{eq:natinj}
\iota_{S_2}:\quad S_2\to \mathbb{Q}_2\otimes_{\mathbb{Z}_2} S_2 = 
\prod_{j=1}^{d-1}\mathbb{Q}_2(\zeta_{2^j})
\end{equation}
which sends $x$ in $S_2$ to the tuple $(\zeta_{2^j})_{j=1}^{d-1}$.
\end{bigprop}

\begin{proof}
Let $J=\langle *\rangle$ be the group of order two generated by the involution $*$. Then $J$ acts on $\mathbb{Z}_2X$ and we have
an exact sequence of $J$-modules
$$0 \to \mathbb{Z}_2 \cdot \Tr_X \to \mathbb{Z}_2X \to S_2 \to 0$$
in which $J$ acts trivially on $\mathbb{Z}_2 \cdot \Tr_X$.  Taking $J$-cohomology and using that $\HH^1(J, \mathbb{Z}_2 \cdot \Tr_X) = 0$ since $J$ acts trivially on 
$\mathbb{Z}_2\cdot \Tr_X \cong \mathbb{Z}_2$ and $\mathbb{Z}_2 $ is torsion free, we obtain that $(\mathbb{Z}_2X)^J \to S_2^J = S_2^+$ is surjective.
Note that $(\mathbb{Z}_2X)^J$ is a free $\mathbb{Z}_2$-module generated by 
$$\{x^j + x^{-j}\}_{j = 1}^{2^{d-2} -1} \cup \{1, x^{2^{d-2}}\}.$$  
Let $\pi:\mathbb{Z}_2[t]\to S_2^+$ be the $\mathbb{Z}_2$-algebra homomorphism given by $\pi(t)=x+x^{-1}$.
The recurrence relation (\ref{eq:pj}) gives immediately that
$\pi(p_j)=x^j+x^{-j}$ for all $j\ge 0$. Moreover, the definition of $\tau$ in Proposition \ref{prop:betaliftpowerseries}
shows that 
\begin{equation}
\label{eq:moretau}
\pi(1-\tau) = x^{2^{d-2}}
\end{equation} 
in $S_2^+$. Define $\Phi(t)\in\mathbb{Z}[t]$ to be the product of the irreducible monic polynomials $q_j(t)$ over $\mathbb{Z}$
of the numbers $\zeta_{2^j} + \zeta_{2^j}^{-1}$ when $\zeta_{2^j}$ is a root of unity of order $2^j$ and $j = 1,\ldots, d-1$. 
Since $S_2$ can be identified with the image of the injective $\mathbb{Z}_2$-algebra
homomorphism $\iota_{S_2}$ from (\ref{eq:natinj}), it follows that $\pi(\Phi(t))=0$.
Using that $q_1(t) = t + 2$, $q_2(t) = t$ and $q_{j+1}(t) = q_j(t^2 - 2)$ if $j \ge 2$, we see that $\Phi(t)$ is a distinguished polynomial
of degree $2^{d-2}$. This implies that $\pi$ induces a well-defined surjective $\mathbb{Z}_2$-algebra homomorphism
$\pi:\mathbb{Z}_2[[t]] \to S_2^+$ such that the kernel of $\pi$ contains the ideal of $\mathbb{Z}_2[[t]]$ generated by $\Phi(t)$.
Since $\mathbb{Z}_2[[t]]/(\Phi(t))$ is a free $\mathbb{Z}_2$-module of rank $2^{d-2}$ and the rank of $\mathbb{Q}_2\otimes_{\mathbb{Z}_2}S_2^+$ 
is also $2^{d-2}$, it follows that $\Phi(t)$ generates the kernel of $\pi$.

Letting $a = 1+2^{d-4}\,t$, it follows from  Proposition \ref{prop:betaliftpowerseries} that the equation
$$b^2 + tab + a^2 = 1-\tau$$
has a solution $b$ in $\mathbb{Z}_2[[t]]$. Applying the homomorphism $\pi$ from (\ref{eq:pi}) to $a$ and $b$ we obtain an 
element
$$\beta=\pi(a)+x\,\pi(b)$$
in $S_2$. Since $\beta^*=\pi(a)+x^{-1}\pi(b)$, we have
\begin{eqnarray*}
\beta\beta^* &=& \pi(a)^2 + \pi(a)\pi(b)(x + x^{-1}) + \pi(b)^2\\
&=&\pi(b^2 + tab + a^2)\,.
\end{eqnarray*}
Using (\ref{eq:moretau}), it follows that $\beta$ is an element of $S_2$ satisfying the relation (\ref{eq:beta}).

For the last statement of Proposition \ref{prop:betalift}, we use the injective 
$\mathbb{Z}_2$-algebra homomorphism $\iota_{S_2}$ 
from (\ref{eq:natinj}). By Proposition \ref{prop:betaliftpowerseries}, 
$\Delta(t) = t^2 ( (1-t)^2-8)$ if $d=4$, and if $d\ge 5$ then
$\Delta(t)=t^2(1+4 \,m(t) )$ for some polynomial $m(t)\in\mathbb{Z}[t]$ whose constant coefficient is divisible by $2$.

If $d=4$ then the inverse of $\pi(1-t)=1-x-x^{-1}$ in
$S_2$ is given by
$$u_1=\frac{1}{3}\,(-1-2x-x^2+x^3+2x^4+x^5-x^6-2x^7)\,.$$
This means that
$$\pi(\Delta(t)) = (x+x^{-1})^2\,(1-x-x^{-1})^2\left(1-8\,u_1^2\right) \quad
\mbox{ for $d=4$}.$$
Therefore, for all $d\ge 4$, 
$$\pi(\Delta(t)) = s^2\,(1+4r)$$
for explicitly given elements $r,s\in S_2$, where $r$ is in the Jacobson radical of $S_2$.
Let $(1 + 4r)_j$ be the image of $1 + 4r$ in $\mathbb{Q}_2(\zeta_{2^j})$ under the
injection $\iota_{S_2}$ in (\ref{eq:natinj}).
Then we can take $\sqrt{1 + 4r} \in S_2$ to be the unique element whose component in 
$\mathbb{Q}_2(\zeta_{2^j})$ is the square root of $(1 + 4r)_j$ which is congruent to 
$1$ modulo $2 \,\mathfrak{m}_j$ when $\mathfrak{m}_j = \mathbb{Z}_2[\zeta_{2^j}](1 - \zeta_{2^j})$ is the maximal ideal of the ring 
of integers $\mathbb{Z}_2[\zeta_{2^j}]$.
\end{proof}

\begin{bigprop}
\label{prop:unidefQ}
Let $R=W\,\overline{\Q}$. Let  $\beta\in S_2$ be as in Proposition $\ref{prop:betalift}$,
let $\gamma$ be the reduction of $\beta$ modulo $2$, and let $L=kX/(k\cdot \Tr_X)$ be the corresponding
$k\Q$-module as in Remark $\ref{rem:endotrivQ}(a)$. Then $L$ is an endo-trivial $k\Q$-module as in Remark $\ref{rem:endotrivQ}(b)$.

If $V_0=k$ is the trivial simple $k\Q$-module, let $V_{0,W}=W$ with trivial $\Q$-action.
If $V_1=L$, let $V_{1,W}=WX/(W\cdot \Tr_X)$ be the $W\Q$-module on
which $y$ acts as
\begin{equation}
\label{eq:yaction}
y\cdot v_1 = \beta \,v_1^*\qquad\mbox{for all $v_1\in V_{1,W}$.}
\end{equation}
Let $i\in\{0,1,2,3\}$ and let $j\in\{0,1\}$. Then 
$\Omega^i_{W\Q}(V_{j,W})$, as defined in Remark 
$\ref{rem:syzygylift}$, is a lift of $\Omega_{\Q}^i(V_j)$ over $W$. Moreover, the universal deformation of $\Omega_{\Q}^i(V_j)$ is given by the isomorphism class of 
the $R\Q$-module $U(\Q,\Omega_{\Q}^i(V_j))=\Omega^i_{W\Q}(V_{j,W})\otimes_W R$ on which $x$ $($resp. $y$$)$ acts diagonally as multiplication by 
$x\otimes\overline{x}$ $($resp. $y\otimes\overline{y}$$)$.
\end{bigprop}

\begin{proof}
The first statement follows from Remark \ref{rem:endotrivQ} and Proposition \ref{prop:betalift}.
To check that  (\ref{eq:yaction}) defines a $W\Q$-module structure on $V_{1,W}$, we follow the corresponding argument in 
the proof of \cite[Lemma 6.4]{CT}. 
The remaining statements of Proposition \ref{prop:unidefQ} now follow  from Theorem \ref{thm:supermain} and Remark \ref{rem:syzygylift}.
\end{proof}


\section{Location of the endo-trivial modules in the stable Auslander-Reiten quiver of $k\D$}
\label{s:appendix}
\setcounter{equation}{0}
Let $k$ be a field of characteristic $2$, and let $\D$ be either $\SD$ or $\Q$. 
In this section, we determine the location of the isomorphism classes of all indecomposable endo-trivial 
$k\D$-modules in the stable Auslander-Reiten quiver of $k\D$. To do so, we first use a slight variation
of results in \cite{bondarenkodrozd} and \cite{Dade1} to obtain an explicit isomorphism
between $k\D$ and a $k$-algebra $\Lambda_{\D}$ of the form $k\langle a,b\rangle/I_{\D}$ such that the
socle scalars match the descriptions in \cite[Thm. III.1 and Sect. III.13]{erd}. 
Then we identify the $\Lambda_{\D}$-modules that correspond to indecomposable endo-trivial $k\D$-modules.
In particular, this provides a different approach to describing the indecomposable endo-trivial
$k\D$-modules, using quivers and relations.

\subsection{Location of the endo-trivial modules in the stable Auslander-Reiten quiver of $k\SD$}
\label{ss:SDchar2}

Let $d\ge 4$ be a fixed integer, and let $\SD$ be the semidihedral group of order $2^d$
given by (\ref{eq:semid}). In \cite[Sect. 3]{bondarenkodrozd}, Bondarenko and Drozd stated an explicit isomorphism
between $k\SD/\mathrm{soc}(k\SD)$ and a certain quotient algebra of $k\langle a, b\rangle$.
The next lemma is a slight variation of their result to obtain an explicit isomorphism for $k\SD$ itself. 
In particular, we obtain precise socle scalars. We provide a proof for the convenience of the reader.

\begin{lemma}
\label{lem:kSD}
Let $\Lambda_{\SD} = k\langle a,b\rangle /I_{\SD}$, where
\begin{equation}
\label{eq:ideal}
I_{\SD}=\left((ab)^{2^{d-2}} - (ba)^{2^{d-2}} , a^2 - b(ab)^{2^{d-2}-1}-(ab)^{2^{d-2}} , b^2, (ab)^{2^{d-2}} a\right).
\end{equation}
Define the following elements in the radical of $k\SD$:
\begin{eqnarray}
\label{eq:a}
r_a&=& (z+yx) + (x+x^{-1}) + \sum_{i=1}^{2^{d-4}-1} \left(x^{4i+1}+x^{-(4i+1)}\right) (1+zy)\,,\\
\label{eq:b}
r_b&=&1+y\,.
\end{eqnarray}
Then the map
\begin{equation}
\label{eq:f}
f_{\SD}:\Lambda_{\SD}\to k\SD\qquad\mbox{defined by $f_{\SD}(a)=r_a$, $f_{\SD}(b)=r_b$}
\end{equation}
induces a $k$-algebra isomorphism between $\Lambda_{\SD}$ and $k\SD$.
\end{lemma}

\begin{proof}
Define 
$$u=z+\sum_{i=0}^{2^{d-2}-1}x^{2i}.$$
The elements $r_a$ and $(u\,r_b)$ are the same elements that were given in \cite[Sect. 3]{bondarenkodrozd} to provide an explicit isomorphism between
\begin{equation}
\label{eq:Lambdasoc}
\overline{\Lambda}_{\SD}=\Lambda_{\SD}/\mathrm{soc}(\Lambda_{\SD}) \cong  k\langle a,b\rangle / 
\left(a^2 - b(ab)^{2^{d-2}-1}, a^3, b^2\right)
\end{equation}
and $k\SD/\mathrm{soc}(k\SD)$. 
Note that $u$ lies in the center of $k\SD$ and $u^2=1$. Since in all the relations for $\Lambda_{\SD}$ 
(and $\overline{\Lambda}_{\SD}$), $b$
occurs an even number of times, the factor $u$ does not play a role. 
Since
$$(z+yx) + (x+x^{-1}) 
=(1+yx) + x^{-1}(1+x)^2+ (1+x)^{2^{d-2}}$$
and since, for $i\ge 1$,
$$\left(x^{4i+1}+x^{-(4i+1)}\right) (1+zy) = x^{-(4i+1)}\left(1+x^{4i+1}\right)^2(1+zy)$$
lies in $\mathrm{rad}^3(k\SD)$, we obtain
\begin{eqnarray}
\label{eq:arel1}
r_a&\equiv& (1+yx) + (x+x^{-1}) \mod \mathrm{rad}^3(k\SD)\\
&\equiv& 1+yx \mod \mathrm{rad}^2(k\SD)\nonumber
\end{eqnarray}
A straightforward  calculation shows that
$$r_a^2 = \sum_{i=0}^{2^{d-2}-1} x^{2i}(1+y).$$
Using (\ref{eq:arel1}) and 
$\mathrm{rad}^{2^{d-1}+1}(k\SD)=0$,
it follows that
\begin{equation}
\label{eq:abrel1}
r_b(r_ar_b)^{2^{d-2}-1} = (1+y)\left[(1+yx + x+x^{-1})(1+y)\right]^{2^{d-2}-1}.
\end{equation}
We have
$$(1+y)(1+yx + x+x^{-1})(1+y)=(x^{-1}+xz)(1+y),$$
and hence, inductively, for all $i\ge 1$,
\begin{equation}
\label{eq:abrel2}
(1+y)\left[(1+yx + x+x^{-1})(1+y)\right]^i=(x^{-1}+xz)^i(1+y).
\end{equation}
Notice that for all $0\le i\le 2^{d-2}-1$, the binomial coefficient ${2^{d-2}-1}\choose i$ is an odd integer. Hence,
$$(x^{-1}+xz)^{2^{d-2}-1} =xz\,(1+x^2z)^{2^{d-2}-1}
=\sum_{i=0}^{2^{d-2}-1}x^{2i+1}.$$
Therefore, we obtain
$$r_b(r_ar_b)^{2^{d-2}-1} = \sum_{i=0}^{2^{d-2}-1}x^{2i+1}(1+y).$$
A straightforward calculation now shows that 
$$(r_ar_b)^{2^{d-2}} =\sum_{i=0}^{2^{d-1}-1}x^i(1+y)= (r_br_a)^{2^{d-2}},$$
and hence $r_a^2=r_b(r_ar_b)^{2^{d-2}-1} +(r_ar_b)^{2^{d-2}}$.
Since $r_b^2=0=(r_ar_b)^{2^{d-2}}r_a$, the map $f_{\SD}$ from (\ref{eq:f}) defines a $k$-algebra homomorphism.
Since $\Lambda_{\SD}$ and $k\SD$ are Artinian local $k$-algebras with residue field $k$ and since the images of 
$r_a$ and $r_b$ generate $\mathrm{rad}(k\SD)/\mathrm{rad}^2(k\SD)$ by (\ref{eq:b}) and (\ref{eq:arel1}), 
we see that $f_{\SD}$ is surjective. Since the $k$-dimensions of $\Lambda_{\SD}$ and $k\SD$ are both 
$2^d$, it follows that $f_{\SD}$ is a $k$-algebra isomorphism.
\end{proof}

The following lemma realizes the $k\SD$-modules $Y$ and $L$ from Remark \ref{rem:endotrivSD}
as $\Lambda_{\SD}$-modules. Since all indecomposable endo-trivial $k\SD$-modules are in the
$\Omega$-orbit of either $L$ or of the trivial simple $k\SD$-module $k$, this provides descriptions
of all $\Lambda_{\SD}$-modules corresponding to indecomposable endo-trivial $k\SD$-modules.

\begin{lemma}
\label{lem:YLSD}
Write $\Lambda=\Lambda_{\SD}$, and let $f_{\SD}:\Lambda\to k\SD$ 
be the $k$-algebra isomorphism from Lemma $\ref{lem:kSD}$. 
Define the following $\Lambda$-modules:
\begin{equation}
\label{eq:LaSD}
Y_\Lambda=\Lambda b \quad\mbox{and}\quad L_a=\Lambda ab.
\end{equation}
Then $Y_\Lambda\cong \Lambda/\Lambda b$ and $L_a\cong \Lambda a/\Lambda a^2\cong\Lambda/\Lambda a$. Moreover,
$Y_\Lambda$ and $L_a$ are uniserial $\Lambda$-modules of length $2^{d-1}$ and $2^{d-1}-1$, respectively.
If $Y$ and $L$ are the $k\SD$-modules from $(\ref{eq:YL})$, then
$Y$ corresponds under $f_{\SD}$ to $Y_\Lambda$ and $L$ corresponds to $L_a$.
\end{lemma}

\begin{proof}
We first prove the statements about the $\Lambda$-modules $Y_\Lambda$ and $L_a$. Because $b^2=0$, it follows that
$\Lambda b$ is a uniserial $\Lambda$-module and that $\mathrm{rad}(\Lambda b)=\Lambda ab$. In other words,
$L_a=\mathrm{rad}(Y_\Lambda)$, and hence $L_a$ is also a uniserial $\Lambda$-module. Considering the surjective 
$\Lambda$-module homomorphism $\Lambda \to \Lambda b$, which sends $1$ in $\Lambda$ to $b$, we see that 
$\Lambda b$ lies in the kernel. Comparing $k$-dimensions, it follows that  $\Lambda/\Lambda b\cong \Lambda b$.
Therefore,
$$L_a\cong\mathrm{rad}(\Lambda/\Lambda b)=\mathrm{rad}(\Lambda)/\Lambda b = (\Lambda a+\Lambda b)/\Lambda b
\cong \Lambda a/\left(\Lambda a\cap \Lambda b\right)=\Lambda a/ \Lambda a^2.$$ 
Considering the surjective $\Lambda$-module homomorphism $\Lambda \to \Lambda ab$, which sends 
$1$ in $\Lambda$ to $ab$, we see that $\Lambda a$ lies in the kernel since $a^2b=0$ in $\Lambda$. Comparing 
$k$-dimensions, it follows that  $\Lambda/\Lambda a\cong \Lambda ab$.

To realize the $k\SD$-modules $Y$ and $L$ as $\Lambda$-modules, we use that 
$$Y=k\,\SD/\langle y\rangle\cong k\SD(1+y)$$ 
and that there is a short exact sequence of $k\SD$-modules
$$0\to k\SD(1+y) \to k\SD \to k\,\SD/\langle y\rangle \to 0\;.$$
Since $f_{\SD}(b)=1+y$ by Lemma \ref{lem:kSD}, this implies that
$Y$ corresponds under $f_{\SD}$ to the $\Lambda$-module $\Lambda/\Lambda b\cong\Lambda b$, i.e. to $Y_\Lambda$.
Since $L=\mathrm{rad}(Y)$, this implies that $L$ corresponds to $L_a$.
\end{proof}

To identify the components of the stable Auslander-Reiten quiver of $k\SD$ to which
the indecomposable endo-trivial $k\SD$-modules belong, we use that the stable Auslander-Reiten quiver of $\Lambda_{\SD}$ is the same
as the Auslander-Reiten quiver of $\overline{\Lambda}_{\SD}=\Lambda_{\SD}/\mathrm{soc}(\Lambda_{\SD})$ in (\ref{eq:Lambdasoc}). 
Using the results from \cite[Sects. II.9.5 and II.10]{erd} together with Lemma \ref{lem:YLSD}, we obtain the following result.

\begin{lemma}
\label{cor:ARSD}
Let $\Theta$ be the component of the stable Auslander-Reiten quiver of $k\SD$ containing $[\Omega^1_{\SD}(k)]$. Then
$\Theta$ is a non-periodic component of type $\mathbb{Z}D_\infty$ such that $[L]$ belongs to $\Theta$.
Moreover, both $[\Omega^{-1}_{\SD}(k)]$ and $[L]$ each have exactly one predecessor in $\Theta$, given in
both cases by $[\mathrm{rad}(k\SD)/\mathrm{soc}(k\SD)]$. In other words, all isomorphism classes
of indecomposable endo-trivial $k\SD$-modules lie in the two components $\Theta$ and $\Omega(\Theta)$, and they lie in four 
$\Omega^2$-orbits of vertices at the ends of these components. 
\end{lemma}

\subsection{Location of the endo-trivial modules in the stable Auslander-Reiten quiver of $k\Q$}
\label{ss:Qchar2}

Let $d\ge 3$ be a fixed integer. Let $\Q$ be the quaternion group of order 8 given by (\ref{eq:quat8}) when $d=3$,
and let $\Q$ be the generalized quaternion group of order $2^d$ given by (\ref{eq:quat}) when $d\ge 4$. 
In \cite[Sect. 1]{Dade1}, Dade gave an explicit isomorphism between $k\Q$ and a certain quotient algebra of 
$k\langle a, b\rangle$, which in particular provides precise socle scalars. 
Note that if $d=3$ then Dade needed to assume that $k$ contains a primitive
cube root of unity for his isomorphism to work. 
Using Dade's explicit isomorphisms together with ideas from \cite[III.10]{erd}, 
we obtain the following lemma, which provides
precise socle scalars that match the descriptions in \cite[Thm. III.1 and Sect. III.13]{erd}.

\begin{lemma} 
\label{lem:kQ}
Let $\Lambda_{\Q} = k\langle a,b\rangle /I_{\Q}$, where
\begin{eqnarray}
\label{eq:idealQ}
I_{\Q}&=&\left((ab)^{2^{d-2}} - (ba)^{2^{d-2}} , a^2 - b(ab)^{2^{d-2}-1} - \delta\,(ab)^{2^{d-2}} , 
 \right.\\
&& \nonumber
\left.b^2- a(ba)^{2^{d-2}-1} - \delta\,(ab)^{2^{d-2}} , (ab)^{2^{d-2}} a\right)
\qquad\mbox{ and }\\
\nonumber
\delta&=&\left\{\begin{array}{c@{\quad\mbox{if}\quad}l}0&d=3\\1&d\ge 4\end{array}\right. .
\end{eqnarray}
If $d=3$, assume $k$ contains a primitive cube root of unity $\omega$, and define the following elements in the radical of $k\Q$:
\begin{eqnarray}
\label{eq:aQ8}
r_a&=& (1+x) + \omega\,(1+yx) + \omega^2\,(1+y),\\
\label{eq:bQ8}
r_b&=& (1+x) + \omega^2\,(1+yx) + \omega\,(1+y).
\end{eqnarray}
If $d\ge 4$, define the following elements in the radical of $k\Q$:
\begin{eqnarray}
\label{eq:r}
r&=&(yx+y)^{2^{d-1}-3} +\sum_{i=1}^{d-3} \left(yx+y\right)^{2^{d-2}-2^i},\\
\label{eq:aQ}
r_a&=& (1+yx + r)+ \left[(1+yx+r)(1+y+r)\right]^{2^{d-2}-1},\\
\label{eq:bQ}
r_b&=&\;(1+y \;+ r)+ \left[(1+yx+r)(1+y+r)\right]^{2^{d-2}-1}.
\end{eqnarray}
Then the map
\begin{equation}
\label{eq:fQ}
f_{\Q}:\Lambda_{\Q}\to k\Q\qquad\mbox{defined by $f_{\Q}(a)=r_a$, $f_{\Q}(b)=r_b$}
\end{equation}
induces a $k$-algebra isomorphism between $\Lambda_{\Q}$ and $k\Q$.
\end{lemma}

The following lemma realizes the $k\Q$-modules in the $\Omega$-orbit of $L$ from Remarks
\ref{rem:endotrivQ8}(b) and \ref{rem:endotrivQ} as $\Lambda_{\Q}$-modules.
Since all indecomposable endo-trivial $k\Q$-modules are in the
$\Omega$-orbit of either $L$ or of the trivial simple $k\Q$-module $k$, this provides descriptions
of all $\Lambda_{\Q}$-modules corresponding to indecomposable endo-trivial $k\Q$-modules.

\begin{lemma}
\label{lem:LQ}
Write $\Lambda=\Lambda_{\Q}$, and let $f_{\Q}:\Lambda\to k\Q$ be the $k$-algebra isomorphism from Lemma $\ref{lem:kQ}$, where we assume that $k$ contains a primitive cube root of unity $\omega$ when $d=3$. 
Define the following $\Lambda$-modules:
\begin{equation}
\label{eq:LabQ}
L_a=\Lambda ab\quad\mbox{and}\quad L_b=\Lambda ba.
\end{equation}
Then $L_a\cong \Lambda/\Lambda a$ and $L_b\cong \Lambda/\Lambda b$
and they are both uniserial $\Lambda$-modules of length $2^{d-1}-1$ whose stable endomorphism rings are isomorphic to $k$. 
The $\Omega$-orbit of $L_a$ is as follows:
\begin{equation}
\label{eq:OmegaLa}
\Omega^1_{\Lambda}(L_a) \cong \Lambda a\cong\Lambda/\Lambda ba;\quad 
\Omega^2_{\Lambda}(L_a)\cong L_b;
\quad
\Omega^3_{\Lambda}(L_a) \cong \Lambda b\cong\Lambda/\Lambda ab;\quad 
\Omega^4_{\Lambda}(L_a)\cong L_a.
\end{equation}
Moreover, $L_a$ and $L_b$ lie at the end of a $2$-tube in the stable Auslander-Reiten quiver of $\Lambda$.
The endo-trivial $k\Q$-module $L$ from Remarks $\ref{rem:endotrivQ8}(b)$ and $\ref{rem:endotrivQ}$
corresponds under $f_{\Q}$ to either $L_a$ or $L_b$, and the $\Omega$-orbit of $L$ corresponds to the
$\Omega$-orbit of $L_a$.
\end{lemma}

\begin{proof}
We first prove the statements about the $\Lambda$-modules $L_a$ and $L_b$. Because of the symmetry of the relations
in $\Lambda=\Lambda_{\Q}$ in $a$ and $b$, it suffices to consider $L_a=\Lambda ab$. Since $a^2b=0$ in $\Lambda$, it follows
that $\Lambda ab$ is a uniserial $\Lambda$-module of length $2^{d-1}-1$. Considering the surjective $\Lambda$-module 
homomorphism $\Lambda \to \Lambda ab$, which sends $1$ in $\Lambda$ to $ab$, we see that $\Lambda a$ lies in the kernel.
Comparing $k$-dimensions, it follows that  $\Lambda/\Lambda a\cong \Lambda ab$. On the other hand, considering the
surjective $\Lambda$-module homomorphism $\Lambda \to \Lambda a$, which sends $1$ in $\Lambda$ to $a$, we see that 
$\Lambda ba$ lies in the kernel since $ba^2=0$ in $\Lambda$. Comparing $k$-dimensions, it follows that  
$\Lambda/\Lambda ba\cong \Lambda a$. This proves in particular (\ref{eq:OmegaLa}). 

Because of the $\Lambda$-module structure of $L_a=\Lambda ab$, we see that its endomorphism ring $\mathrm{End}_\Lambda(L_a)$ 
has $k$-dimension $2^{d-2}$, with a $k$-basis given by the endomorphisms 
$g_{a,i}:L_a\to L_a$, with $g_{a,i}(ab) =(ab)^{i}$, for  $1\le i\le 2^{d-2}$.
Since $g_{a,1}$ is the identity map and $g_{a,i}$
factors through $\Lambda$ for $2\le i \le 2^{d-2}$, it follows that the stable endomorphism ring 
$\underline{\mathrm{End}}_\Lambda(L_a)$ is isomorphic to $k$. 

By (\ref{eq:OmegaLa}), $L_a$ and $L_b$ lie in a 2-tube of the stable Auslander-Reiten quiver of $\Lambda$.
Note that 
$$\mathrm{Ext}^1_\Lambda(L_a,L_b)\cong
\underline{\mathrm{Hom}}_\Lambda(\Omega^1_\Lambda(L_a),L_b)\cong
\underline{\mathrm{Hom}}_\Lambda(\Lambda a,\Lambda ba)\cong k.$$
This can be seen by noticing that a $k$-basis of $\mathrm{Hom}_\Lambda(\Lambda a,\Lambda ba)$ is given by the homomorphisms
$g_{ab,i}:\Lambda a\to \Lambda ba$, with $g_{ab,i}(a) =a(ba)^{i}$, for $1\le i \le 2^{d-2}-1$,
together with $g_{ab,0}:\Lambda a\to \Lambda ba$, with $g_{ab,0}(a)= (ba)^{2^{d-2}}$. Since $g_{ab,0}$ does not
factor through a projective $\Lambda$-module and  $g_{ab,i}$ factors through $\Lambda$ for $1\le i\le 2^{d-2}-1$, it follows
that $\mathrm{Ext}^1_\Lambda(L_a,L_b)\cong k$. One non-zero extension of $L_a$ by $L_b$ is given by 
the indecomposable $\Lambda$-module
$$M_{ab}=\frac{\left(\Lambda/(\Lambda ba+\Lambda a^2)\right)\oplus \left(\Lambda/\Lambda b\right)}{k\,(a+(\Lambda ba+\Lambda a^2),(ba)^{2^{d-2}-1}+\Lambda b)}\;.$$
Therefore, there must exist an almost split sequence of the form
$$0\to L_b \to M_{ab}\to L_a\to 0.$$
Since $M_{ab}$ is indecomposable, it follows that $L_a$ and $L_b$ lie at the end of a 2-tube.

To realize the endo-trivial $k\Q$-module $L$ as a $\Lambda$-module, we use (\ref{eq:OmegaLa}) and that we know the
$k$-dimensions of the indecomposable endo-trivial $k\Q$-modules by Remarks \ref{rem:endotrivQ8}(b) and \ref{rem:endotrivQ}. Namely, the
$k$-dimensions of the $k\Q$-modules in the $\Omega$-orbit of the trivial $k\Q$-module $k$ are
$$\mathrm{dim}_k(k)=1;\quad \mathrm{dim}_k \,\Omega^1_{\Q}(k)=2^d-1;\quad
\mathrm{dim}_k \,\Omega^2_{\Q}(k)=2^d+1; \quad\mathrm{dim}_k \,\Omega^3_{\Q}(k)=2^d-1.$$
On the other hand, the $k$-dimensions of the $\Lambda$-modules in the $\Omega$-orbit of $L_a$ are by (\ref{eq:OmegaLa})
$$\mathrm{dim}_k(L_a)=2^{d-1}-1=\mathrm{dim}_k \,\Omega^2_\Lambda(L_a);\quad 
\mathrm{dim}_k \,\Omega^1_\Lambda(L_a)=2^{d-1}+1=\mathrm{dim}_k \,\Omega^3_\Lambda(L_a).$$
Therefore, $L$ corresponds under $f_{\Q}$ to either $L_a$ and $L_b$, and hence the $\Omega$-orbit of $L$ corresponds to the
$\Omega$-orbit of $L_a$.
\end{proof}

Finally, we identify the components of the stable Auslander-Reiten quiver of $k\Q$ to which  the
indecomposable endo-trivial $k\Q$-modules belong. Using Remark \ref{rem:endotrivQ8}(a),
\cite[Lemma IV.1.10]{erd} and Lemma \ref{lem:LQ}, we obtain the following result.

\begin{lemma}
\label{cor:ARQ}
If $d=3$ and $k$ does not contain a primitive cube root of unity, then the 4 isomorphism classes of 
indecomposable endo-trivial $k\Q$-modules lie at the ends of two 2-tubes of
the stable Auslander-Reiten quiver of $k\Q$, given by
$$\{[k],[\Omega^2_{\Q}(k)]\};\quad\{[\Omega^{-1}_{\Q}(k)],[\Omega^1_{\Q}(k)]\}.$$

Otherwise, there are 8 isomorphism classes of indecomposable
endo-trivial $k\Q$-modules, and they lie at the ends of four 2-tubes of
the stable Auslander-Reiten quiver of $k\Q$, given by
$$\{[k],[\Omega^2_{\Q}(k)]\};\quad\{[\Omega^{-1}_{\Q}(k)],[\Omega^1_{\Q}(k)]\};\quad 
\{[L],[\Omega^2_{\Q}(L)]\};\quad\{[\Omega^{-1}_{\Q}(L)],[\Omega^1_{\Q}(L)]\}.$$
\end{lemma}



\begin{thebibliography}{88}

\bibitem{alp} J.~L.~Alperin, Local representation theory. Modular representations as an introduction to the local representation theory of finite groups. Cambridge Studies in Advanced Mathematics, vol. 11, Cambridge University Press, Cambridge, 1986.

\bibitem{alpendotrivial} J.~L.~Alperin, Lifting endo-trivial modules. J. Group Theory 4 (2001), 1--2.

\bibitem{bl} F.~M.~Bleher, Universal deformation rings and Klein four defect groups. Trans. Amer. Math. Soc. 354 (2002), 3893--3906.

\bibitem{bleherTAMS2009} F.~M.~Bleher, Universal deformation rings and dihedral defect groups. Trans. Amer. Math. Soc. 361 (2009), 3661--3705.
	
\bibitem{diloc} F.~M.~Bleher, Universal deformation rings for dihedral $2$-groups. 
J. London Math Soc. 79, (2009), 225--237. 

\bibitem{3quat} F.~M.~Bleher, Universal deformation rings and generalized quaternion defect 
groups. Adv. Math. 225 (2010), 1499--1522.

\bibitem{bc} F.~M.~Bleher and T.~Chinburg, Universal deformation rings and cyclic blocks. Math. Ann. 318 (2000), 805--836.

\bibitem{bondarenkodrozd} V.~M.~Bondarenko and Y.~A.~Drozd, The representation type of finite groups. 
Zap. Nau\v{c}n. Sem. LOMI 71 (1977), 24--41, translation in J. Soviet Math. 20 (1982) 2515--2528.

\bibitem{brouepuig} M.~Brou\'{e} and L.~Puig, A Frobenius theorem for blocks.  
Invent. Math.  56  (1980), 117--128.

\bibitem{carl1} J.~F.~Carlson,  A characterization of endotrivial modules over $p$-groups.  
Manuscripta Math.  97  (1998),   303--307.

\bibitem{carlsonpapers} J.~F.~Carlson,  N.~Mazza and D.~K.~Nakano,
Endotrivial modules for finite groups of Lie type A in nondefining characteristic.
Math. Z. 282 (2016), no. 1-2, 1--24. 

\bibitem{CT} J.~F.~Carlson and J.~Th\'{e}venaz, Torsion endo-trivial modules.  Algebr. Represent. 
Theory  3  (2000),  303--335. 

\bibitem{carl2} J.~F.~Carlson and J.~Th\'{e}venaz, The classification of endo-trivial modules.  
Invent. Math.  158  (2004),  389--411.

\bibitem{carl1.5} J.~F.~Carlson and J.~Th\'{e}venaz, The classification of torsion endo-trivial modules.  
Ann. of Math. (2)  162  (2005),  823--883.

\bibitem{CR} C.~W.~Curtis and I.~Reiner,  Methods of representation theory. Vol. I. With applications to finite groups and orders. John Wiley \& Sons, Inc., New York, 1981. 

\bibitem{Dade1} E.~Dade, Une extension de la th\'{e}orie de Hall et Higman. J. Algebra
20 (1972), 570--609.

\bibitem{dade} E.~Dade, Endo-permutation modules over $p$-groups. I; II. Ann. of Math. 107 (1978), 459--494; 108 (1978), 317--346.

\bibitem{lendesmit} B.~de Smit and H.~W.~Lenstra, Explicit construction of universal deformation rings. In: Modular Forms and 
Fermat's Last Theorem (Boston, MA, 1995), Springer-Verlag, Berlin-Heidelberg-New York, 1997, pp. 313--326.

\bibitem{erd} K.~Erdmann, Blocks of Tame Representation Type and Related Algebras. Lecture Notes in Mathematics, vol. 1428, 
Springer-Verlag, Berlin-Heidelberg-New York, 1990.

\bibitem{green} J.~A.~Green, A lifting theorem for modular representations. Proc. Roy. Soc. London. Ser. A 252 (1959), 135--142. 

\bibitem{linckel} M.~Linckelmann, A derived equivalence for blocks with dihedral defect groups. J. Algebra 164 (1994), 244--255.

\bibitem{linckel2} M.~Linckelmann, The source algebras of blocks with a Klein four defect group.  J. Algebra  167  (1994),  821--854. 

\bibitem{linckelstable} M.~Linckelmann, Stable equivalences of Morita type for self-injective algebras and 
$p$-groups. Math. Z. 223 (1996), 87--100.

\bibitem{linckel3} M.~Linckelmann, On stable equivalences of Morita type.  Derived equivalences for group rings,  221--232, Lecture Notes in Math., 1685, Springer, Berlin, 1998.
       
\bibitem{maz1} B.~Mazur, Deforming Galois representations. In: Galois groups over $\mathbb{Q}$ (Berkeley, CA, 1987), Springer-Verlag, Berlin-Heidelberg-New York, 1989, pp. 385--437.

\bibitem{malle} C.~Lassueur, G.~Malle and E.~Schulte, Simple endotrivial modules for quasi-simple groups. J. reine angew. Math. 712 (2016), 141--174.

\bibitem{puig} L.~Puig, Nilpotent blocks and their source algebras.  Invent. Math.  93  (1988),  77--116.        

\bibitem{puig2} L.~Puig, On the local structure of Morita and Rickard equivalences between Brauer blocks. 
Progress in Mathematics, 178. Birkh\"auser Verlag, Basel, 1999.

\bibitem{serre} J.~P.~Serre, Linear representations of finite groups.  Graduate Texts  in Mathematics, vol. 42, 
Springer-Verlag, Berlin-Heidelberg-New York, 1977.

\bibitem{soto} R.~C.~Soto, Universal deformation rings and semidihedral 2-groups.
Dissertation, University of Iowa, 2015.

\bibitem{thevenaz} J.~Th\'{e}venaz, $G$-algebras and modular representation theory. Oxford University Press, Oxford, 1995.

\end{thebibliography}
\end{document}